\documentclass[aop,seceqn,MSNbibl,citesort,dvips]{arximspdf}

\doi{10.1214/10-AOP588}
\volume{39}
\issue{4}
\pubyear{2011}
\firstpage{1422}
\lastpage{1448}

\makeatletter

\newtheorem{theorem}{Theorem}[section]
\newtheorem{corollary}[theorem]{Corollary}
\newtheorem{lemma}[theorem]{Lemma}

\newproclaim{Example}{Example}

\makeatother

\begin{document}
\begin{frontmatter}

\title{Backward stochastic dynamics on a filtered probability space\thanksref{T1}}
\runtitle{Backward stochastic dynamics}

\thankstext{T1}{Supported in part by EPSRC Grant EP/F029578/1 and by the Oxford-Man Institute.}

\begin{aug}
\author[A]{\fnms{Gechun} \snm{Liang}\ead[label=e1]{liangg@maths.ox.ac.uk}},
\author[A]{\fnms{Terry} \snm{Lyons}\ead[label=e2]{tlyons@maths.ox.ac.uk}} and
\author[A]{\fnms{Zhongmin} \snm{Qian}\corref{}\ead[label=e3]{qianz@maths.ox.ac.uk}}
\runauthor{G. Liang, T. Lyons and Z. Qian}
\affiliation{University of Oxford}
\address[A]{Mathematical Institute\\
University of Oxford\\
Oxford OX1 3LB\\
United Kingdom\\
and\\
Oxford-Man Institute\\
University of Oxford\\
Oxford OX2 6ED\\
United Kingdom\\
\printead{e1}\\
\phantom{E-mail: }\printead*{e2}\\
\phantom{E-mail: }\printead*{e3}} 
\end{aug}

\received{\smonth{7} \syear{2009}}
\revised{\smonth{3} \syear{2010}}

%
\begin{abstract}
We demonstrate that backward stochastic differential equations
(BSDE) may be reformulated as ordinary functional differential
equations on certain path spaces. In this framework, neither
It\^{o}'s integrals nor martingale representation formulate are
needed. This approach provides new tools for the study of BSDE, and
is particularly useful for the study of BSDE with partial
information. The approach allows us to study the following type of
backward stochastic differential equations:
\[
dY_{t}^{j}=-f_{0}^{j}(t,Y_{t},L(M)_{t})\,dt-%
\sum_{i=1}^{d}f_{i}^{j}(t,Y_{t})\,dB_{t}^{i}+dM_{t}^{j}
\]
with $Y_{T}=\xi$, on a general filtered probability space $(\Omega,%
\mathcal{F},\mathcal{F}_{t},\mathbf{P})$, where $B$ is a
$d$-dimensional Brownian motion, $L$ is a prescribed (nonlinear)
mapping which sends a square-integrable $M$ to an adapted process
$L(M)$ and $M$, a~correction term, is a square-integrable
martingale to be determined. Under certain
technical conditions, we prove that the system admits a unique solution
$%
(Y,M)$. In general, the associated partial differential equations
are not only nonlinear, but also may be nonlocal and involve
integral operators.
\end{abstract}

%

\end{frontmatter}

\section{Introduction}\label{intro}

Stochastic differential equations (SDE) may be considered as
dynamical systems perturbed by random signals which are often
modeled by Brownian motion. The important class of stochastic
differential equations considered in the literature are It\^{o}-type
equations such as
%
\begin{equation} \label{july-1}
dX_{t}^{j}=f_{0}^{j}(t,X_{t})\,dt+\sum
_{i=1}^{d}f_{i}^{j}(t,X_{t})\,dB_{t}^{i},
\end{equation}
where $B=(B^{1},\ldots,B^{d})$ is Brownian motion in
$\mathbf{R}^{d}$
on a completed probability space $(\Omega,\mathcal{F},\mathbf{P})$, $%
f_{i}=\sum_{j=1}^{d^{\prime}}f_{i}^{j}\frac{\partial}{\partial
x^{j}}$ are bounded, smooth vector fields in
$\mathbf{R}^{d^{\prime}}$, $j=1,\ldots,d^{\prime}$, where
$d$, $d^{\prime}$ are two positive integers. It\^{o} gave the
meaning of solutions to (\ref{july-1}) by developing a theory of
stochastic integration against Brownian motion (called It\^{o}'s
calculus), and obtained strong solutions by specifying an initial
data at a starting time $T$.

SDE (\ref{july-1}) has to be interpreted as an integral equation
\[
X_{t}^{j}-X_{0}^{j}=\int_{0}^{t}f_{0}^{j}(s,X_{s})\,ds+\sum_{i=1}^{d}%
\int_{0}^{t}f_{i}^{j}(s,X_{s})\,dB_{s}^{i},
\]
which can be solved forward (i.e., for $t>0$). It\^{o}'s calculus
requires
that a solution $X=(X_{t})$ has to be adapted to Brownian motion $%
B=(B^{1},\ldots,B^{d})$; it is thus not necessarily possible to solve
(\ref%
{july-1}) backward from a certain time $T$ to $t<T$.

There are interesting applications on the other hand to be able to
solve (%
\ref{july-1}) backward. Suppose $u$ is a smooth solution to the
Cauchy problem of the quasi-linear parabolic equation
\[
\biggl( \frac{1}{2}\Delta-\frac{\partial}{\partial t}\biggr)
u+f(u,\nabla u)=0\qquad\mbox{on }[0,\infty)\times\mathbf{R}^{d},
\]
with the initial data $u(x,0)=u_{0}(x)$. Let $T>0$ and
$h(t,x)=u(T-t,x)$ for $t\in[ 0,T]$. Then $h$ solves the
backward parabolic equation
\[
\biggl( \frac{1}{2}\Delta+\frac{\partial}{\partial t}\biggr)
h+f(h,\nabla h)=0\qquad\mbox{on }[0,T]\times\mathbf{R}^{d},
\]
and $h(x,T)=u_{0}(x)$. Let $Y_{t}=h(t,B_{t})$ where $B$ is Brownian
motion in $\mathbf{R}^{d}$. According to It\^{o}'s formula
%
\begin{equation} \label{p-r1}
Y_{T}-Y_{t}=\int_{t}^{T}\biggl( \frac{\partial}{\partial s}+\frac{1}{2}%
\Delta\biggr) h(s,B_{s})\,ds+M_{T}-M_{t}
\end{equation}
for $t\leq T$, where $M_{t}=\int_{0}^{t}\nabla h(s,B_{s})\,dB_{s}$ is
a
martingale. Substituting $( \frac{\partial}{\partial s}+\frac{1}{2}%
\Delta) h$ by $-f_{0}(h,\nabla h)$ in (\ref{p-r1}) obtains
%
\begin{equation} \label{p-r1a}
Y_{T}-Y_{t}=-\int_{t}^{T}f(Y_{s},\nabla
h(s,B_{s}))\,ds+M_{T}-M_{t}.
\end{equation}
According to It\^{o}'s martingale representation theorem, the
\textit{density process} $Z_{t}=\nabla h(t,B_{t})$ of $M$ with respect
to Brownian motion is uniquely determined as the unique predictable
process $Z_{t}$ such that
\[
M_{T}-M_{0}=\sum_{j=1}^{d}\int_{0}^{T}Z_{t}^{j}\,dB_{t}^{j}.
\]
In terms of the pair $(Y,Z)$ (\ref{p-r1a}) may be written as
\[
Y_{T}-Y_{t}=-\int_{t}^{T}f(Y_{s},Z_{s})\,ds+\sum_{j=1}^{d}%
\int_{t}^{T}Z_{s}^{j}\,dB_{s}^{j}
\]
with the terminal data $Y_{T}=u_{0}(B_{T})$, which is the integral
form of the following backward stochastic differential equation:
%
\begin{equation}\label{p-r-3}
dY_{t}=-f(Y_{t},Z_{t})\,dt+Z_{t}\,dB_{t},\qquad Y_{T}=\xi,
\end{equation}
introduced and studied by Pardoux and Peng \cite{MR1037747}.

In the past twenty years, there has been a large number of articles
devoted to the theory of BSDE and its applications in various
research areas. Our references listed at the end of the paper are by
no means complete, and the
reader should refer to excellent surveys such as articles in \cite%
{MR1752671} edited by El Karoui and Mazliak, the recent paper by
El Karoui, Hamadene and Matoussi \cite{ElKaroui},
the book by Yong and Zhou \cite{MR1696772}
and the references therein for a guide
to the BSDE literature.

To the knowledge of the present authors, it was Bismut \cite{Bismut}
(see \cite{MR0453161,MR0469466}) who first formulated
terminal problems for a class of stochastic differential equations
in order to study stochastic optimal control problems by means of
Pontryagin's maximum principal. His equations, called backward
stochastic differential equations, have been extended and developed
to a nonlinear case in the seminal paper \cite{MR1037747} by
Pardoux and Peng. A lot of efforts have been made to generalize the
class of BSDE considered in \cite{MR1037747}. For example, Lepeltier
and San Martin \cite{MR1602231} relaxed the Lipschitz continuous
condition on the driver and studied BSDEs with coefficients of
linear growth. Yong \cite{MR1440146} employed the continuity method
to prove the existence of solution with arbitrary time horizon. In
\cite{Briand-re1} Briand et al. considered $L^{p}$-solutions for
BSDE. It is also natural to consider BSDE coupled with a forward
stochastic differential equation,
called a forward--backward stochastic differential equations. Antonelli
\cite%
{MR1233625} first studied such FBSDE; his equation does not involve a
density process $Z$ in the driver. A definite account about FBSDE
may be found in Ma, Protter and Young \cite{MR1262970}, Hu and Peng
\cite{MR1355060}, Peng and Wu \cite%
{MR1675098}, the recent book \cite{MR1704232} and the literature
therein. Most authors consider BSDE on a probability space with
Brownian filtration, and there are a few papers dealing with BSDE
with jumps or with reflecting boundary conditions. Tang and Li
\cite{MR1243240} have studied BSDE with random jumps, and
Barles, Buckdahn and Pardoux \cite{MR1436432} have explored the connection between BSDE with
random jumps and some parabolic integro-differential equations. Rong
\cite{MR1440399} proved the existence and uniqueness under
non-Lipschitz continuous coefficients for this class of BSDE.
Analogous to free-boundary PDE problems, El Karoui et al.
\cite{MR1434123} introduced an obstacle to BSDE such that the
solution always stays above such obstacle. This so-called reflected
BSDE is further developed to double reflected barriers by
Cvitani{\'{c}} and Karatzas \cite{MR1415239} and
Hamadene, Lepeltier and Matoussi
\cite{MR1752681}. Furthermore, Bally, Pardoux and Stoica \cite{MR2127730} have
considered BSDE on the probability space associated with Dirichlet
processes.

If the driver of BSDE is with quadratic growth of $Z$, the nature of
equations is completely changed. This problem is first solved by
Kobylanski \cite{MR1782267} by using the Cole--Hopf transformation adopted
from the PDE theory. Her results have been substantially developed
and generalized by Briand and Hu \cite{MR2257138,MR2391164},
where they extend to equations with convex drivers subject to
unbounded terminal values. Most of the existing literature
concentrates on solutions of BSDEs in a strong sense, that is, the
underlying filtered probability space is given. One of the first
attempts to introduce weak solutions for BSDEs was presented in
Buckdahn, Engelbert and R{\u{a}}{\c{s}}canu
\cite{MR2141331}, and Buckdahn and Engelbert
\cite{MR2354579} who further proved the uniqueness of their weak
solutions, while the coefficients of their BSDEs do not evolve a
density process $Z$. On the other hand, the notion of weak solution
for FBSDEs was introduced by
Antonelli and Ma \cite{MR1978231} and further developed by
Ma, Zhang and Zheng
\cite%
{MR2478677} by employing the martingale problem approach.

The backward stochastic differential equations have found many
connections with other research areas: stochastic control, PDE,
mathematical finance, etc. To derive a maximum principle as
necessary conditions for optimal control problems, one can observe
that the adjoint equations to the optimal control problems satisfy
certain backward equations. For stochastic control problems, the
corresponding adjoint equations are stochastic rather than
deterministic. Indeed Peng \cite{MR105163} established a general
stochastic maximum principle by considering both first-order and
second-order adjoint equations, and, on the other hand, Kohlmann and
Zhou \cite{MR1766421}
interpreted BSDE as equivalent to stochastic control problems. Peng
\cite%
{MR1149116} and Pardoux and Pend~\cite{MR1176785} derived a probabilistic representation (a
\textit{Feynman--Kac representation}) for solutions of some
quasi-linear PDEs, which was extended to other cases by
Ma, Protter and Yong
\cite{MR1262970}. The later has been summarized as a four-step
scheme of solving forward--backward stochastic differential
equations (FBSDE) (see \cite{MR1704232} by Ma and Yong for details).
Cheridito, Soner, Touzi and Victoir \cite{Touzi} connected
a class of second order BSDEs to fully nonlinear PDEs.
In \cite{Duffie} Duffie and Epstein discovered a class of nonlinear BSDE in
their study of recursive utility in economics. Later El Karoui, Peng
and Quenez \cite{MR1434407} applied BSDE to option pricing problems and
provided a general framework for the application of BSDE in finance. In
order to deal with utility maximization problems in incomplete markets,
Rouge and El Karoui~\cite{MR1802922} introduced a class of BSDE with
quadratic growth. Hu, Imkeller and M{\"u}ller \cite{MR2152241} further studied this class
of BSDE in a more general setting.

In this article, we put forward a simple approach to deal with the
kind of BSDE such as (\ref{p-r-3}) which does not depend on any
martingale representation, and thus allows us to study a wide class of
backward stochastic dynamics. Our main idea and contribution in this
article is to establish an
\textit{ordinary functional differential equation} which is equivalent to
(\ref%
{p-r-3}), which allows us to obtain alternative representations for
solutions of BSDE and to consider a new interesting class of
stochastic dynamical systems.

Consider the following example of backward stochastic differential
equations:%
%
\begin{equation} \label{a-b2}
dY_{t}=-f(t,Y_{t},Z_{t})\,dt+\sum_{i=1}^{d}Z_{t}^{i}\,dB_{t}^{i},
\qquad
Y_{T}=\xi,
\end{equation}
where $B=(B_{t})_{t\geq0}$ is Brownian motion in
$\mathbf{R}^{d}$, $\xi
\in L^{2}(\Omega,\mathcal{F}_{T},\mathbf{P})$ and $(\mathcal{F}%
_{t})_{t\geq0}$ is the Brownian filtration. The differential
equation has
to be interpreted as the integral equation%
%
\begin{equation} \label{a-b3}
\xi
-Y_{t}=-\int_{t}^{T}f(s,Y_{s},Z_{s})\,ds+\sum_{i=1}^{d}%
\int_{t}^{T}Z_{s}^{i}\,dB_{s}^{i}.
\end{equation}
By applying the Picard iteration to $(Y,Z)$, one shows that if $f$
is globally Lipschitz continuous, then there is a unique pair
$(Y,Z)$ which satisfies (\ref{a-b3}) for all $t\leq T$. This method
relies on the martingale representation for Brownian motion and
thus restricts the class of BSDE.

Our main idea is based on the following simple observation. Suppose
that $%
Y=(Y_{t})_{t\in[ \tau,T]}$ is a solution of (\ref{a-b3})
back to time $\tau<T$, then $Y$ must be a special semimartingale
whose variation part is continuous. Let $Y_{t}=M_{t}-V_{t}$ be the
Doob--Meyer
decomposition into its martingale part $M$ and its finite variation
part~$-V$%
. The decomposition over $[\tau,T]$ is unique up to a random
variable measurable with respect to $\mathcal{F}_{\tau}$. Let us
assume that the local martingale part $M$ is indeed a martingale up
to $T$. Then, since the
terminal value $Y_{T}=\xi$ is given, $\xi=M_{T}-V_{T}$, so that
$M_{t}=%
\mathbf{E}(\xi+V_{T}|\mathcal{F}_{t})$ and
$Y_{t}=\mathbf{E}(\xi+V_{T}|\mathcal{F}_{t})-V_{t}$ for $t\in
[ \tau,T]$. The integral equation (\ref{a-b3}) thus can be
written as
\[
\xi
-M_{t}+V_{t}=-\int_{t}^{T}f(s,Y_{s},Z_{s})\,ds+\sum_{i=1}^{d}%
\int_{t}^{T}Z_{s}^{i}\,dB_{s}^{i}
\]
for every $t\in[ \tau,T]$. Taking expectations, with both sides
conditional on $\mathcal{F}_{t}$, one obtains
\begin{eqnarray*}
\mathbf{E}(\xi|\mathcal{F}_{t})-M_{t}+V_{t}
&=&-\mathbf{E}\biggl[ \int_{\tau
}^{T}f(s,Y_{s},Z_{s})\,ds\Big\vert\mathcal{F}_{t}\biggr]
\\
&&{}+\int_{\tau}^{t}f(s,Y_{s},Z_{s})\,ds.
\end{eqnarray*}
By identifying the martingale parts and variational parts, we must
have
%
\begin{equation} \label{fde1}
V_{t}-V_{\tau}=\int_{\tau}^{t}f(s,Y_{s},Z_{s})\,ds,
\end{equation}
where $Y$ and $Z$ are considered as functionals of $V$, namely%
%
\begin{equation} \label{a-b7}
Y_{t}=\mathbf{E}(\xi+V_{T}|\mathcal{F}_{t})-V_{t},\qquad M_{t}=%
\mathbf{E}(\xi+V_{T}|\mathcal{F}_{t}),
\end{equation}
and $Z$ is determined uniquely by the martingale representation through%
\[
M_{T}-M_{\tau}=\sum_{i=1}^{d}\int_{\tau
}^{T}Z_{s}^{i}\,dB_{s}^{i}.
\]
Hence $Y$ and $Z$ are written as $Y(V)$ and $Z(V)$, respectively, if
we wish to emphasize the fact that $Y$ and $Z$ are defined entirely
through $V$. Observe that (\ref{fde1}) is clearly the integral form
of the functional
differential equation%
\[
\frac{dV}{dt}=f(t,Y(V)_{t},Z(V)_{t}),
\]
which can be solved by Picard iteration applying to $V$ alone,
rather than the pair $(Y,Z)$.

The approach may be made independent of the use of a martingale
representation theorem, provided that one is willing to replace the
density process $Z$ by a functional of $V$, thus freeing us from the
requirement of Brownian filtration. This kind of generalization of
BSDE theory is a bit surprising and even overly rewarded,
which is, however, not the only point we would like to emphasize. More
precisely, we may consider the correction martingale part appearing
in (\ref{a-b2}) as part of the solution rather than its density
process $Z$. That is, by setting $M_{t}-M_{\tau
}=\sum_{i=1}^{d}\int_{\tau}^{t}Z_{s}^{i}\,dB_{s}^{i}$, and regarding
$Z$ as a function of $M$, so denoted by~$L(M)$, then (\ref{a-b2})
can be reformulated
as%
%
\begin{equation}\label{a-b5}
dY_{t}=-f(t,Y_{t},L(M)_{t})\,dt+dM_{t},\qquad Y_{T}=\xi,
\end{equation}
which is in turn equivalent to the functional integral equation
%
\begin{equation}\label{a-b8}
V_{t}-V_{\tau}=\int_{\tau}^{t}f(s,Y(V)_{s},L(M(V))_{s})\,ds,
\end{equation}
where $Y(V)$ and $M(V)$ are given by (\ref{a-b7}). For (\ref{a-b8}),
there is no need to insist that $L$ sends a martingale $M$ to its
density process (if there is any), though the density process
mapping $L$ remains the most interesting case.

The approach might be applied to a more general setting of solving
dynamical systems backward under other constraints, not necessarily
the adaptedness to a filtration; even a probability setting is not
necessary. One possible example can be the following. One may study
the functional differential equation (\ref{fde1}), where
$Y\dvtx V\rightarrow Y(V)$ and $M\dvtx V\rightarrow M(V)$ are defined in terms
of some kind of ``projections''
instead of conditional expectations. We, however, in this paper, make
no attempt for such an extension.

To our knowledge, most of BSDE which currently exist in the
literature may be studied in the framework of ordinary functional
differential equations. Since our approach does not rely on the
martingale representation theorem, we are able to study a class of
BSDE on an arbitrary filtered probability space. We, however, would
like to point out that this paper is not so much about generalizing
the theory of BSDE to a general filtered probability space; our main
contribution is the equivalence of BSDE and a class of ordinary
functional integral equations. We allow a sufficient wide class of
functionals $L(M)$ which, even in the classical setting, extends the
associated PDE to some nonlocal integro-differential equations.

If $(\mathcal{F}_{t})_{t\geq0}$ is Brownian filtration, any martingale
$%
M=(M_{t})_{t\geq0}$ has an It\^{o} integral representation $%
M_{t}-M_{0}=\sum_{j=1}^{d}\int_{0}^{t}Z_{s}^{j}\,dW_{s}^{j}$ which
determines the density $Z=(Z^{1},\ldots,Z^{d})$. Consider the
functional over
martingales%
\[
L(M)_{t}=\mathbf{E}\biggl\{ \int_{t}^{T}Z_{s}\mu(t,ds)\Big|\mathcal{F}%
_{t}\biggr\},
\]
where $\mu(t,ds)$ is a transition kernel (not random for
simplicity), and
consider the corresponding BSDE%
\[
dY_{t}^{j}=-f_{0}^{j}\bigl(T-t,Y_{t},L(M)_{t}\bigr)\,dt+dM_{t}^{j},\qquad Y_{T}=\xi.
\]
Our approach demonstrates the existence and uniqueness for
this kind of BSDE, whose associated PDE is a system of
integro-differential
equations,%
\[
\frac{\partial}{\partial t}u-\frac{1}{2}\Delta u+f_{0}(t,u,H(u))=0,
\]
where the nonlinear operator $H$ involves space--time integration, and
indeed%
\[
H(u)(t,x)=\int_{t}^{T}\int_{\mathbf{R}^{d}}\frac{\nabla
u(T-s,z)}{(2\pi(s-t))^{d/2}}e^{-{|x-z|^{2}}/({2(s-t)})}\,dz\,\mu
(t,ds).
\]
If $\mu(t,ds)=\delta_{t}(ds)$ then we recover the case considered
in the current literature. By choosing different functionals $L(M)$
we may obtain even more general integro-differential equations. This
kind of integro-differential equations often appears in the study of
particle limiting models for PDE; one class of equations which has
a similar nature is already in the literature, for example, in
Majda \cite{majda1}.

In this paper we constrain ourselves to the study of the following
type of backward stochastic differential equations:
%
\begin{equation}\label{ss-c}
dY_{t}^{j}=-f_{0}^{j}(t,Y_{t},L(M)_{t})\,dt-%
\sum_{i=1}^{d}f_{i}^{j}(t,Y_{t})\,dB_{t}^{i}+dM_{t}^{j},
\end{equation}
subject to $Y_{T}=\xi$, on a filtered probability space $(\Omega,%
\mathcal{F},\mathcal{F}_{t},\mathbf{P})$, where $B$ is a
$d$-dimensional Brownian motion as given, $j=1,\ldots,d^{\prime}$,
$L$ is a given
(nonlinear) functional on square-integrable martingales, while
$(\mathcal{F}%
_{t})_{t\geq0}$ is not necessary to be Brownian filtration. A solution
to (%
\ref{ss-c}) is a pair $(Y,M)$, where $Y=(Y^{j})$ are semimartingales
and $%
M=(M^{j})$ are square-integrable martingales, which satisfies the
corresponding integral equations:
%
\begin{eqnarray} \label{ss-b}
Y_{T}^{j}-Y_{t}^{j}&=&-\int_{t}^{T}f_{0}^{j}(t,Y_{t},L(M)_{t})\,dt-%
\sum_{i=1}^{d}\int
_{t}^{T}f_{i}^{j}(t,Y_{t})\,dB_{t}^{i}\nonumber\\[-8pt]\\[-8pt]
&&{}+M_{T}^{j}-M_{t}^{j}.\nonumber
\end{eqnarray}
The term $L(M)$ appearing in the drift term $f_{0}$ on the
right-hand side of (\ref{ss-c}) suggests that $L$ is a mapping which
sends a vector of square-integrable martingales $M=(M^{j})$ to a
progressively measurable process $L(M)$. The backward stochastic
equation (\ref{ss-c}) is thus described by the driver $f_{0}$, the
diffusion coefficients $f_{i}$ together with the prescribed mapping
$L$.

Finally, let us point out that similar ideas have been known in the
PDE theory. Recall that, for any reasonable function $u$, $u$ has
the following decomposition:
\[
u=H(u)+G(u),
\]
where $H(u)$ is a harmonic function determined by a boundary
integral against a Green function, and $G(u)$ is a potential. Thus
the boundary condition (which corresponds to our case the terminal
value) determines the harmonic function part~$H(u)$. The regularity
theory for nonlinear PDE such as $\Delta u=f(u,\nabla u)$ may be
developed via the previous decomposition,
by studying the Newtonian potential $G(u)$, (Gilbarg and Trudinger \cite
{MR1814364}). In this way, backward stochastic dynamics, as a class
of Markov processes, can be regarded as a generic extension of some
nonlinear PDE problems of finite dimension to infinite-dimensional
problems in path spaces. On the other hand, some nonlinear PDE can
be considered as a pathwise version of backward stochastic dynamics.
We will explore these ideas further in coming papers.

The paper is organized as follows. In Section~\ref{sec2} we present some
elementary facts and basic assumptions. The main results of the
existence of local and global solutions, and the uniqueness of the
backward stochastic dynamics are presented and proved in Sections
\ref{sec3}
and~\ref{sec4}.

\section{Preliminaries}\label{sec2}

Let $(\Omega,\mathcal{F},\mathcal{F}_{t},\mathbf{P})$ (where
$t\in[ 0,\infty)$) be a filtered probability space which
satisfies the \textit{usual conditions}: $(\Omega
,\mathcal{F},\mathbf{P})$ is a completed probability space,
$( \mathcal{F}_{t}) _{t\geq0}$ is a
right-continuous filtration, each $\mathcal{F}_{t}$ contains all events
in $%
\mathcal{F}$ with probability zero and $\mathcal{F}=\sigma\{\mathcal{F}
_{t}\dvtx t\geq0\}$. Under these technical assumptions, any martingale on $%
(\Omega,\mathcal{F},\mathcal{F}_{t},\mathbf{P})$ has a
modification whose sample paths are right continuous with left-hand
limits. Henceforth, by a martingale we always mean a martingale
which is right continuous with left-hand limits.

Let $0\leq\tau<T$ be any but fixed numbers. $[\tau,T]$ serves as
the region of the time parameter, although we are working with a
fixed filtered probability space $(\Omega
,\mathcal{F},\mathcal{F}_{t},\mathbf{P})$. Let
$\mathcal{C}( [ \tau,T] ;\mathbf{R}^{d})
$ denote the space of all continuous, adapted processes
$(V_{t})_{t\in[ \tau,T]}$ valued in $\mathbf{R}^{d}$
such that ${\max_{j}\sup_{t\in
[ \tau,T]}}|V_{t}^{j}|$ belongs to $L^{2}(\Omega$, $\mathcal{F}_{T},%
\mathbf{P})$, equipped with the norm
\[
\|V\|_{\mathcal{C}[\tau
,T]}=\sqrt{\sum_{j=1}^{d}\mathbf{E}\sup_{t\in[ \tau
,T]}|V_{t}^{j}|^{2}}.
\]
$\mathcal{C}( [ \tau,T] ;\mathbf{R}^{d})
$ is a
Banach space under \mbox{$\|\cdot\|_{\mathcal{C}[\tau,T]}$}, $\mathcal{M}%
^{2}([\tau,T];\mathbf{R}^{d})$ denotes the space of $\mathbf{R}^{d}$%
-valued square-integrable martingales on $(\Omega,\mathcal{F},\mathcal
{F}%
_{t},\mathbf{P})$ from time $\tau$ up to time $T$ (which, of
course,
can be uniquely extended to a martingale in $\mathcal{M}^{2}([0,T],%
\mathbf{R}^{d})$), together with the norm $\|M\|_{\mathcal{C}[\tau
,T]}$%
. We also need the direct sum space of $\mathcal{M}^{2}([\tau,T];%
\mathbf{R}^{d})$ and $\mathcal{C}( [ \tau,T] ;%
\mathbf{R}^{d}) $, denoted by $\mathcal{S}([\tau,T];\mathbf{R}%
^{d})$. If $Y\in\mathcal{S}([\tau,T];\mathbf{R}^{d})$, its
decomposition into an element in $\mathcal{M}^{2}([\tau,T];\mathbf{R}%
^{d})$ and the other in $\mathcal{C}( [ \tau,T] ;%
\mathbf{R}^{d}) $ may not be unique, and there are various
norms one can define on $\mathcal{S}([\tau,T];\mathbf{R}^{d})$.
For our
purposes, we choose the norm $\|Y\|_{\mathcal{C}[\tau,T]}$, although $%
\mathcal{S}([\tau,T];\mathbf{R}^{d})$ is not complete under
\mbox{$\|\cdot
\|_{\mathcal{C}[\tau,T]}$}. Finally $\mathcal{H}^{2}([\tau
,T];\mathbf{R}%
^{d^{\prime}\times d})$ denotes the space of all \textit{predictable}
processes $Z=(Z_{t}^{j,i})_{t\in[ \tau,T]}$ on $(\Omega,\mathcal{F},%
\mathcal{F}_{t},\mathbf{P})$ with running time $[\tau,T]$,
endowed with the usual $L^{2}$-norm
\[
\|Z\|_{\mathcal{H}_{[\tau,T]}^{2}}=\sqrt{\sum_{j=1}^{d^{\prime
}}\sum_{i=1}^{d}\mathbf{E}\int_{\tau
}^{T}|Z_{s}^{i,j}|^{2}\,ds}.
\]

If $Y$ is a semimartingale on $(\Omega,\mathcal{F},\mathcal{F}_{t},%
\mathbf{P})$ over time interval $[\tau,T]$ with its Doob--Meyer
decomposition $Y_{t}=M_{t}-V_{t}$, such that $M$ is an
$\mathcal{F}_{t}$-martingale during $[\tau,T]$, $V$ is a
\textit{continuous}, adapted process with finite variation on
$[\tau,T]$ and $V_{T}$, $Y_{T}$ are integrable,
then $M_{t}=\mathbf{E}(Y_{T}+V_{T}|\mathcal{F}_{t})$ and $Y_{t}=%
\mathbf{E}(Y_{T}+V_{T}|\mathcal{F}_{t})-V_{t}$ for $t\in
[ \tau
,T]$. Since we are interested in terminal value problems, in which $%
Y_{T}=\xi$ are given, therefore, for given $\xi=(\xi^{i})$ where
$\xi^{i}\in L^{2}(\Omega,\mathcal{F}_{T},\mathbf{P})$, we
consider two
functionals on $\mathcal{C}( [ \tau,T] ;\mathbf{R}%
^{d}) \dvtx V\rightarrow Y(V)$ and $V\rightarrow M(V)$ defined
by
%
\begin{equation} \label{y-r1}
Y(V)_{t}=\mathbf{E}(\xi+V_{T}|\mathcal{F}_{t})-V_{t}\qquad\mbox{for }%
t\in[ \tau,T]
\end{equation}
and
%
\begin{equation} \label{y-r2}
M(V)_{t}=\mathbf{E}( \xi+V_{T}|\mathcal{F}_{t})
\qquad\mbox{for }t\in[ \tau,T]
\end{equation}
for any $V\in\mathcal{C}( [ \tau,T] ;\mathbf{R}%
^{d}) $. If we wish to indicate the dependence on the terminal value $%
\xi$ as well, then we will use $Y(\xi,V)$ and $M(\xi,V)$ in places
of $%
Y(V)$ and $M(V)$, respectively.

Note that $(Y(V)_{t})_{t\in[ \tau,T]}$ does not depend on
the initial value $V_{\tau}$, an important fact we will use in our
construction of global solutions for the terminal value problem
(\ref{ss-c}).

We consider the following type of backward stochastic differential
equations:
%
\begin{equation}\label{bsde-a}
dY_{t}^{j}=-f_{0}^{j}(t,Y_{t},L(M)_{t})\,dt-%
\sum_{i=1}^{d}f_{i}^{j}(t,Y_{t})\,dB_{t}^{i}+dM_{t}^{j},\qquad
Y_{T}^{j}=\xi^{j},\hspace*{-28pt}
\end{equation}
on the filtered probability space $(\Omega,\mathcal{F},\mathcal{F}_{t},
\mathbf{P})$ ($j=1,\ldots,d^{\prime}$), where $B$ is a
$d$-dimensional Brownian motion on $(\Omega
,\mathcal{F},\mathcal{F}_{t},\mathbf{P})$ as
given, $T>0$ is the terminal time, $\xi^{j}\in L^{2}(\Omega,\mathcal
{F}%
_{T},\mathbf{P})$ (for $j=1,\ldots,d^{\prime}$) are terminal
values, $%
f_{i}^{j}$ ($i=0,\ldots,d$ and $j=1,\ldots,d^{\prime}$) are
locally\vspace*{1pt}
bounded and Borel measurable, and $L$ is a prescribed mapping on
$\mathcal{M}%
^{2}([\tau,T];\mathbf{R}^{d^{\prime}})$ valued in $\mathcal{H}%
^{2}([\tau,T];\mathbf{R}^{d^{\prime}\times d})$ or in $\mathcal{C}
( [ \tau,T] ;\mathbf{R}^{d^{\prime}}) $.

A solution of (\ref{bsde-a}) backward to time $\tau$ is a pair of
adapted
processes $(Y_{t}$, $M_{t})_{t\in[ \tau,T]}$ where $%
M^{j}=(M_{t}^{j})_{t\in[ \tau,T]}$ are square-integrable
martingales and $Y_{t}^{j}=(Y_{t}^{j})_{t\in[ \tau,T]}$ are
special semimartingales with continuous variation parts, which
satisfies the integral equations
%
\begin{eqnarray}\label{ing1}
Y_{t}^{j}-\xi
^{j}&=&\int_{t}^{T}f_{0}^{j}(s,Y_{s},L(M)_{s})\,ds+\sum_{i=1}^{d}%
\int_{t}^{T}f_{i}^{j}(s,Y_{s})\,dB_{s}^{i}\nonumber\\[-8pt]\\[-8pt]
&&{}+M_{t}^{j}-M_{T}^{j}\nonumber
\end{eqnarray}
for $t\in[ \tau,T]$, $j=1,\ldots,d^{\prime}$.

As we have seen in the \hyperref[intro]{Introduction}, by writing $Y_{t}=M_{t}-V_{t}$, a
solution $%
(Y,M)$ to (\ref{ing1}) is equivalent to a solution $V$ of the
functional
integral equation%
%
\begin{equation} \label{ing2}\qquad
V_{t}-V_{\tau}=\int_{\tau
}^{t}f_{0}(s,Y(V)_{s},L(M(V))_{s})\,ds+\sum_{i=1}^{d}\int_{\tau
}^{t}f_{i}(s,Y(V)_{s})\,dB_{s}^{i},
\end{equation}
where $M(V)_{t}=\mathbf{E}(\xi+V_{T}|\mathcal{F}_{t})$ and $%
Y(V)_{t}=M(V)_{t}-V_{t}$ for $t\in[ \tau,T]$. It is the
integral equation (\ref{ing2}) we are going to study.

The following standard assumptions are always imposed on our backward
SDE (%
\ref{bsde-a}). Additional conditions on $L$ will be introduced later
on to ensure local and global existence.

(1) $f_{0}=(f_{0}^{j})_{j\leq d^{\prime}}$ are Lipschitz continuous on $
[0,\infty)\times\mathbf{R}^{d^{\prime}}\times
\mathbf{R}^{m}$ and $f_{i}=(f_{i}^{j})_{j\leq d^{\prime}}$
($i=1,\ldots,d$) Lipschitz continuous on $[0,\infty)\times
\mathbf{R}^{d^{\prime}}$: there is a constant $C_{2}$ such that
\begin{eqnarray*}
|f_{0}(t,y,z)|&\leq& C_{2}(1+t+|y|+|z|),
\\
|f_{0}(t,y,z)-f_{0}(t,y^{\prime},z^{\prime})|&\leq&
C_{2}(|y-y^{\prime}|+|z-z^{\prime}|),
\\
|f_{i}(t,y)|&\leq& C_{2}(1+t+|y|)
\end{eqnarray*}
and
\[
|f_{i}(t,y)-f_{i}(t,y^{\prime})|\leq C_{2}|y-y^{\prime}|
\]
for $t\geq0$, all $y,y^{\prime}\in\mathbf{R}^{d^{\prime}}$ and
$z$, $%
z^{\prime}\in\mathbf{R}^{m}$.

(2) The terminal value $\xi=(\xi^{i})_{i=1,\ldots,d^{\prime}}$,
$\xi^{i}\in L^{2}(\Omega,\mathcal{F}_{T},\mathbf{P})$.

\section{Local solutions and uniqueness}\label{sec3}

In this section, we prove two results: the uniqueness and the
existence of a local solution to (\ref{bsde-a}) under the assumption
that $L$ is Lipschitz continuous

(3) $L\dvtx\mathcal{M}^{2}([\tau,T];\mathbf{R}^{d^{\prime
}})\rightarrow
\mathcal{H}^{2}([\tau,T];\mathbf{R}^{m})$ (resp., $\mathcal
{C}([\tau,T];%
\mathbf{R}^{m})$):%
\[
\|L(M)-L(\tilde{M})\|_{\mathcal{H}^{2}}\leq C_{1}\|M-\tilde
{M}\|_{\mathcal{C}%
}
\]
[resp.,
\[
\|L(M)-L(\tilde{M})\|_{\mathcal{C}}\leq
C_{1}\|M-\tilde{M}\|_{\mathcal{C}}]
\]
for any $M$, $\tilde{M}\in\mathcal{M}^{2}([\tau,T];\mathbf{R}%
^{d^{\prime}})$, where $\|M\|_{\mathcal{C}}$ means $\|M\|_{\mathcal{C}%
([\tau,T];\mathbf{R}^{d^{\prime}})}$ etc. for
simplicity.

By ``local solution'' we mean a solution from $T$ back to $\tau$, where
$%
T-\tau$ is smaller than a certain constant depending on $L$ and
$f_{i}^{j}$.

In order to prove the uniqueness, we need to consider BSDE in a more
general form than (\ref{bsde-a}). More precisely, we are given
another Brownian
motion $W=(W^{1},\ldots,W^{m^{\prime}})$ on $(\Omega,\mathcal{F}_{t},
\mathcal{F},\mathbf{P})$ and $g_{k}\dvtx\mathbf{R}_{+}\times
\mathbf{R}^{d^{\prime}}\rightarrow\mathbf{R}^{d^{\prime}}$ which are
Lipschitz continuous
\[
|g_{k}(t,y)|\leq C_{2}(1+t+|y|)
\]
and
\[
|g_{k}(t,y)-g_{k}(t,y^{\prime})|\leq C_{2}|y-y^{\prime}|
\]
for all $t\geq0$, $y,y^{\prime}\in\mathbf{R}^{d^{\prime}}$, $%
k=1,\ldots,m^{\prime}$. Define
\begin{eqnarray*}
&&\mathbf{L}\dvtx\mathcal{M}^{2}([\tau,T];\mathbf{R}^{d^{\prime
}})\times
\mathcal{S}([\tau,T];\mathbf{R}^{d^{\prime}})\\
&&\qquad\rightarrow\mathcal
{H}%
^{2}([\tau,T];\mathbf{R}^{m})\qquad\mbox{(resp., }\mathcal{C}([\tau,T];%
\mathbf{R}^{m})\mbox{)}
\end{eqnarray*}
by
%
\begin{equation}\label{e-la1}
\mathbf{L}( M,Y) =L\Biggl( M-\sum_{k=1}^{m^{\prime
}}\int_{\tau}^{\cdot}g_{k}(s,Y_{s})\,dW_{s}^{k}\Biggr) .
\end{equation}
Consider the following mapping $\mathbb{L}$ defined on $\mathcal{C}%
_{0}([\tau,T];\mathbf{R}^{d^{\prime}})$, the space of all
processes in
$\mathcal{C}([\tau,T];\mathbf{R}^{d^{\prime}})$ with initial data
$%
V_{\tau}=0$, by
%
\begin{eqnarray}\label{l-r}
\mathbb{L}(V)_{t} &=&\int_{\tau}^{t}f_{0}(s,Y(V)_{s},\mathbf{L}%
(M(V),Y(V))_{s})\,ds \nonumber\\[-8pt]\\[-8pt]
&&{}+\sum_{i=1}^{d}\int_{\tau}^{t}f_{i}(s,Y(V)_{s})\,dB_{s}^{i},\nonumber
\end{eqnarray}
where $M(V)_{t}=\mathbf{E}(\xi+V_{T}|\mathcal{F}_{t})$ and $%
Y(V)_{t}=M(V)_{t}-V_{t}$ for $t\in[ \tau,T]$, so that $Y(V)_{T}=\xi$%
. As we have seen, the functional integral equation:
$V=\mathbb{L}(V)$,
is equivalent to the following BSDE:%
%
\begin{equation} \label{l-r2}
dY_{t}^{j}=-f_{0}^{j}(t,Y_{t},\mathbf{L}(M,Y)_{t})\,dt-%
\sum_{i=1}^{d}f_{i}^{j}(t,Y_{t})\,dB_{t}^{i}+dM_{t}^{j},\qquad
Y_{T}=\xi.\hspace*{-28pt}
\end{equation}

\begin{lemma}
\label{lem-u1}$\mathbf{L}$ defined by (\ref{e-la1}) is Lipschitz
continuous
%
\begin{eqnarray} \label{e-la2}
&&\|\mathbf{L}( M,Y) -\mathbf{L}(\tilde{M},\tilde{Y})\|_{%
\mathcal{H}^{2}[\tau,T]}\nonumber\\[-8pt]\\[-8pt]
&&\qquad\leq C_{1}\|M-\tilde{M}\|_{\mathcal{C}[\tau,T]}+\frac{m^{\prime
}C_{1}C_{2}}{\sqrt{2}}(T-\tau)\|Y-\tilde{Y}\|_{\mathcal{C}[\tau
,T]}\nonumber
\end{eqnarray}
and
%
\begin{eqnarray} \label{e-laa2}
&&\|\mathbf{L}( M,Y) -\mathbf{L}(\tilde{M},\tilde{Y})\|_{%
\mathcal{C}[\tau,T]} \nonumber\\[-8pt]\\[-8pt]
&&\qquad\leq C_{1}\|M-\tilde{M}\|_{\mathcal{C}[\tau,T]}+2m^{\prime
}C_{1}C_{2}%
\sqrt{T-\tau}\|Y-\tilde{Y}\|_{\mathcal{C}[\tau,T]}\nonumber
\end{eqnarray}
for any $M,\tilde{M}\in\mathcal{M}^{2}([\tau
,T];\mathbf{R}^{d^{\prime
}})$ and $Y,\tilde{Y}\in\mathcal{C}([\tau,T];\mathbf{R}^{d^{\prime}})$.
\end{lemma}
\begin{pf}
Let us omit the subscript $[\tau,T]$ for simplicity. Then
\begin{eqnarray*}
&&\|\mathbf{L}( M,Y) -\mathbf{L}(\tilde{M},\tilde{Y})\|_{%
\mathcal{H}^{2}} \\
&&\qquad\leq C_{1}\|M-\tilde{M}\|_{\mathcal{C}}+C_{1}\sum_{k=1}^{m^{\prime
}}\biggl\Vert\int_{\tau}^{\cdot}\bigl(g_{k}(s,Y_{s})-g_{k}(s,\tilde{Y}%
_{s})\bigr)\,dW_{s}^{k}\biggr\Vert_{\mathcal{H}^{2}} \\
&&\qquad= C_{1}\|M-\tilde{M}\|_{\mathcal{C}}+C_{1}\sum_{k=1}^{m}\sqrt
{\mathbf{E}%
\int_{\tau}^{T}\biggl\vert\int_{\tau}^{t}\bigl(g_{k}(s,Y_{s})-g_{k}(s,\tilde{Y}%
_{s})\bigr)\,dW_{s}^{k}\biggr\vert^{2}\,dt} \\
&&\qquad=C_{1}\|M-\tilde{M}\|_{\mathcal{C}}+C_{1}\sum_{k=1}^{m^{\prime}}\sqrt
{%
\mathbf{E}\int_{\tau}^{T}\int_{\tau}^{t}\bigl\vert
\bigl(g_{k}(s,Y_{s})-g_{k}(s,\tilde{Y}_{s})\bigr)\bigr\vert^{2}\,ds\,dt} \\
&&\qquad\leq C_{1}\|M-\tilde{M}\|_{\mathcal{C}}+m^{\prime}C_{1}C_{2}\sqrt{%
\mathbf{E}\int_{\tau}^{T}\int_{\tau}^{t}\vert Y_{s}-\tilde{Y}%
_{s}\vert^{2}\,ds\,dt} \\
&&\qquad\leq C_{1}\|M-\tilde{M}\|_{\mathcal{C}}+\frac{m^{\prime
}C_{1}C_{2}}{\sqrt{%
2}}(T-\tau)\|Y-\tilde{Y}\|_{\mathcal{C}}.
\end{eqnarray*}
The proof of the second inequality is similar.
\end{pf}

The following is our basic local existence theorem.
\begin{theorem}
\label{lemma3} Under the assumptions on $L$, $f_{j}^{i}$ and
$g_{j}^{i}$
described above. Let%
%
\begin{equation} \label{e-5}
l_{1}=\frac{1}{C_{2}^{2}[ 4C_{1}+6( 1+2\sqrt{d}) +3\sqrt{2}%
m^{\prime}C_{1}C_{2}] ^{2}}\wedge1,
\end{equation}
which depends on the Lipschitz constants $C_{1},C_{2}$ and the
dimensions, but is independent of the terminal data $\xi$. Suppose
that $T-\tau\leq
l_{1}$, then $\mathbb{L}$ admits a unique fixed point on $\mathcal{C}%
_{0}([\tau,T];\mathbf{R}^{d^{\prime}})$.
\end{theorem}
\begin{pf}
The proof is a standard\vspace*{2pt} use of the fixed point theorem applying to
$\mathbb{L%
}$. To this end, we need to show that $\mathbb{L}$ is a contraction on $
\mathcal{C}_{0}([\tau,T];\mathbf{R}^{d^{\prime}})$ as long as
$T-\tau
\leq l_{1}$. This can be done by devising a priori estimates for
$\mathbb{L}$%
. Let us prove the case that $L\dvtx\mathcal{M}^{2}([\tau,T];\mathbf{R}
^{d^{\prime}})\rightarrow\mathcal{H}^{2}([\tau
,T];\mathbf{R}^{m})$ is Lipschitz; the other case can be treated
similarly. To simplify our notation, let $\delta\equiv T-l_{1}$ be
the life duration. Since
\begin{eqnarray*}
\|\mathbb{L}(V)\|_{\mathcal{C}} &\leq&\sqrt{\delta}\sqrt{\mathbf{E}%
\int_{\tau}^{T}|f_{0}(s,Y(V)_{s},\mathbf{L}(M(V),Y(V))_{s})|^{2}\,ds} \\
&&{}+2\sqrt{\sum_{i=1}^{d}\mathbf{E}\int_{\tau
}^{T}|f_{i}(s,Y(V)_{s})|^{2}\,ds}.
\end{eqnarray*}
Since $f_{0}$ and $f_{i}$ are Lipschitz continuous, so that
%
\begin{eqnarray}\label{3-e}
\|\mathbb{L}(V)\|_{\mathcal{C}} &\leq&2C_{2}\bigl( \sqrt{\delta}+\sqrt{d}%
\bigr) \sqrt{\int_{\tau}^{T}(1+s)^{2}\,ds} \nonumber\\
&&{}+2C_{2}\bigl( \sqrt{\delta}+\sqrt{d}\bigr) \sqrt{\int_{\tau}^{T}%
\mathbf{E}|Y(V)_{s}|^{2}\,ds} \\
&&{}+2C_{2}\sqrt{\delta}\|\mathbf{L}(M(V),Y(V))\|_{\mathcal{H}^{2}}.%
\nonumber
\end{eqnarray}
Together with the elementary estimates
\[
\|Y(V)\|_{\mathcal{C}}\leq2\sqrt{\mathbf{E}|\xi
|^{2}}+3\|V\|_{\mathcal{%
C}}
\]
and
\[
\|M(V)\|_{\mathcal{C}}\leq2\sqrt{\mathbf{E}|\xi
|^{2}}+2\|V\|_{\mathcal{%
C}},
\]
one deduces that
%
\begin{eqnarray} \label{st1}\qquad
\|\mathbb{L}(V)\|_{\mathcal{C}} &\leq& \frac{2}{\sqrt{3}}C_{2}\bigl( \sqrt{%
\delta}+\sqrt{3d}\bigr) \delta\sqrt{\delta} \nonumber\\
&&{} +2\bigl[ \sqrt{2}m^{\prime}C_{1}C_{2}^{2}\delta+2C_{2}\sqrt{\delta}%
+2C_{2}\sqrt{d}+2C_{2}C_{1}\bigr] \sqrt{\delta
}\sqrt{\mathbf{E}|\xi
|^{2}} \\
&&{} +\bigl[ 3\sqrt{2}m^{\prime}C_{1}C_{2}^{2}\delta+6C_{2}\sqrt{\delta}%
+4C_{2}C_{1}+6C_{2}\sqrt{d}\bigr] \sqrt{\delta
}\|V\|_{\mathcal{C}}.\nonumber
\end{eqnarray}
Similarly, for $V,\tilde{V}\in\mathcal{C}[\tau,T]$ such that $V_{\tau
}=%
\tilde{V}_{\tau}=0$ one has
\begin{eqnarray*}
&&\|\mathbb{L}(V)-\mathbb{L}(\tilde{V})\|_{\mathcal{C}}
\\
&&\qquad\leq\sqrt{%
\mathbf{E}\biggl( \int_{\tau}^{T}|f_{0}(s,Y_{s},\mathbf{L}%
(M,Y)_{s})-f_{0}(s,\tilde{Y}_{s},\mathbf{L}(\tilde{M},\tilde{Y}%
)_{s})|\,ds\biggr) ^{2}} \\
&&\qquad\quad{}+\sqrt{\mathbf{E}\sup_{t\in[ \tau,T]}\Biggl\vert
\sum_{i=1}^{d}\int_{\tau}^{t}[ f_{i}(s,Y_{s})-f_{i}(s,\tilde{Y}_{s})%
] \,dB_{s}^{i}\Biggr\vert^{2}},
\end{eqnarray*}
where $M_{t}=\mathbf{E}(\xi+V_{T}|\mathcal{F}_{t})$, $\tilde
{M}_{t}=%
\mathbf{E}(\xi+\tilde{V}_{T}|\mathcal{F}_{t})$,
$Y_{t}=M_{t}-V_{t}$ and $\tilde{Y}_{t}=M_{t}-V_{t}$. Since $f_{i}$
are Lipschitz continuous, so that
\begin{eqnarray*}
&&\sqrt{\mathbf{E}\biggl( \int_{\tau}^{T}|f_{0}(s,Y_{s},\mathbf{L}%
(M,Y)_{s})-f_{0}(s,\tilde{Y}_{s},\mathbf{L}(\tilde{M},\tilde{Y}%
)_{s})|\,ds\biggr) ^{2}} \\
&&\qquad\leq C_{2}\sqrt{\mathbf{E}\biggl[ \int_{\tau}^{T}\bigl( |Y_{s}-\tilde{Y%
}_{s}|+\vert\mathbf{L}(M,Y)_{s}-\mathbf{L}(\tilde{M},\tilde{Y}%
)_{s}\vert\bigr) \,ds\biggr] ^{2}} \\
&&\qquad\leq C_{2}\sqrt{\delta}\sqrt{\mathbf{E}\int_{\tau}^{T}\bigl( |Y_{s}-%
\tilde{Y}_{s}|+\vert\mathbf{L}(M,Y)_{s}-\mathbf{L}(\tilde{M},%
\tilde{Y})_{s}\vert\bigr) ^{2}\,ds} \\
&&\qquad\leq C_{2}\delta\|Y-\tilde{Y}\|_{\mathcal{C}}+C_{2}\sqrt{\delta}\|%
\mathbf{L}(M,Y)-\mathbf{L}(\tilde{M},\tilde{Y})\|_{\mathcal{H}^{2}}
\\
&&\qquad\leq C_{2}\biggl[ 1+\sqrt{\delta}\frac{m^{\prime}C_{1}C_{2}}{\sqrt{2}}%
\biggr] \delta\|Y-\tilde{Y}\|_{\mathcal{C}}\\
&&\qquad\quad{}+C_{2}C_{1}\sqrt{\delta}\|M-%
\tilde{M}\|_{\mathcal{C}},
\end{eqnarray*}
where the last inequality follows from (\ref{e-laa2}). It\^{o}'s
integration term can be treated similarly. Applying Doob's
inequality, one has
\begin{eqnarray*}
&&\sqrt{\mathbf{E}\sup_{t\in[ \tau,T]}\Biggl\vert
\sum_{i=1}^{d}\int_{\tau}^{t}[ f_{i}(s,Y_{s})-f_{i}(s,\tilde{Y}_{s})%
] \,dB_{s}^{i}\Biggr\vert^{2}} \\
&&\qquad\leq 2\sqrt{\mathbf{E}\Biggl\vert\sum_{i=1}^{d}\int_{\tau
}^{T}[
f_{i}(s,Y_{s})-f_{i}(s,\tilde{Y}_{s})] \,dB_{s}^{i}\Biggr\vert^{2}} \\
&&\qquad\leq 2C_{2}\sqrt{d}\sqrt{\mathbf{E}\int_{\tau}^{T}|Y_{s}-\tilde
{Y}%
_{s}|^{2}\,ds} \\
&&\qquad\leq 2C_{2}\sqrt{d}\sqrt{\delta
}\|Y-\tilde{Y}\|_{\mathcal{C}}.
\end{eqnarray*}
Putting these estimates together we obtain
%
\begin{eqnarray} \label{k=pr1}
\|\mathbb{L}(V)-\mathbb{L}(\tilde{V})\|_{\mathcal{C}} &\leq&C_{2}\biggl[ 1+%
\sqrt{\delta}\frac{m^{\prime}C_{1}C_{2}}{\sqrt{2}}\biggr] \delta\|Y-%
\tilde{Y}\|_{\mathcal{C}} \nonumber\\[-8pt]\\[-8pt]
&&{}+C_{2}\bigl( C_{1}+2\sqrt{d}\bigr) \sqrt{\delta}\|M-\tilde{M}\|_{%
\mathcal{C}}.\nonumber
\end{eqnarray}
On the other hand it is easy to see that
\begin{eqnarray*}
\|M-\tilde{M}\|_{\mathcal{C}} &=&\sqrt{\mathbf{E}\sup_{t\in
[ \tau
,T]}\mathbf{E}(V_{T}-\tilde{V}_{T}|\mathcal{F}_{t})^{2}} \\
&\leq&2\|V-\tilde{V}\|_{\mathcal{C}}
\end{eqnarray*}
and
\[
\|Y-\tilde{Y}\|_{\mathcal{C}}\leq
3\|V-\tilde{V}\|_{\mathcal{C}}.
\]
Inserting these estimates into (\ref{k=pr1}) we finally obtain
%
\begin{eqnarray}\label{contr-01}
&&\|\mathbb{L}(V)-\mathbb{L}(\tilde{V})\|_{\mathcal{C}}\nonumber\\[-8pt]\\[-8pt]
&&\qquad\leq
C_{2}\biggl[
2C_{1}+6\sqrt{d}+3\sqrt{\delta}+3\delta\frac{m^{\prime
}C_{1}C_{2}}{\sqrt{2%
}}\biggr] \sqrt{\delta}\|V-\tilde{V}\|_{\mathcal{C}}.\nonumber
\end{eqnarray}
Since $\delta\leq l_{1}$, the constant in front of the norm on the
right-hand side is less than $\frac{1}{2}$, so that
\[
\|\mathbb{L}(V)-\mathbb{L}(\tilde{V})\|_{\mathcal{C}}\leq\tfrac
{1}{2}\|V-%
\tilde{V}\|_{\mathcal{C}}.
\]
Therefore $\mathbb{L}$ is a contraction on $\mathcal{C}_{0}([\tau,T];%
\mathbf{R}^{d})$ as long as $T-\tau\leq l_{1}$, so there is a
unique fixed point in $\mathcal{C}_{0}[\tau,T]$. This completes the
proof.
\end{pf}

We are now in a position to show the local existence and uniqueness
of solutions to BSDE (\ref{bsde-a}).
\begin{theorem}
\label{tha1} Let $L$, $f_{j}^{i}$ be Lipschitz continuous with
Lipschitz constants $C_{1}$, $C_{2}$ and
\[
l_{2}=\frac{1}{C_{2}^{2}[ 4C_{1}+6( 1+2\sqrt{d}) +3%
\sqrt{2}d^{\prime}C_{1}C_{2}] ^{2}}\wedge1,
\]
which is\vspace*{1pt} independent of the terminal data $\xi\in L^{2}(\Omega
,\mathcal{F}%
_{T},\mathbf{P})$. Suppose that $T-\tau\leq l_{2}$ and $%
L(M)=L(M-M_{\tau})$ for any $M\in\mathcal{M}^{2}([\tau,T];\mathbf{R}%
^{d^{\prime}})$. Then there is a pair $(Y,M)$, where
$Y=(Y_{t})_{t\in[ \tau,T]}$ is a special semimartingale,
$M=(M_{t})_{t\in[ \tau,T]}$ is a square-integrable
martingale, which solves the backward stochastic differential
equation (\ref{bsde-a}) to time $\tau$. Moreover,
such a pair of solution is unique in the sense that if $(Y,M)$ and
$(\tilde{Y%
},\tilde{M})$ are two pairs of solutions, then $Y=\tilde{Y}$ and
$M-M_{\tau}=\tilde{M}-\tilde{M}_{\tau}$ on $[\tau,T]$.
\end{theorem}
\begin{pf}
By Theorem~\ref{lemma3} (applying to the case that all $g_{k}=0$),
there is a unique $V\in\mathcal{C}_{0}[\tau,T]$ such that
\[
V_{t}=\int_{\tau
}^{t}f_{0}(s,Y_{s},L(M)_{s})\,ds+\sum_{i=1}^{d}\int_{\tau
}^{t}f_{i}(s,Y_{s})\,dB_{s}^{i} \qquad\forall t\in[ \tau
,T],
\]
where $M_{t}=\mathbf{E}( \xi+V_{T}|\mathcal{F}_{t}) $ and $%
Y_{t}=M_{t}-V_{t}$. It is clear that $Y_{T}=\xi$ and
%
\begin{equation} \label{i-eqa1}\qquad
Y_{t}-\xi
=\int_{t}^{T}f_{0}(s,Y_{s},L(M)_{s})\,ds+\sum_{i=1}^{d}%
\int_{t}^{T}f_{i}(s,Y_{s})\,dB_{s}^{i}+M_{t}-M_{T}
\end{equation}
for all $t\in[ \tau,T]$. Therefore $(Y,M)$ solves the
backward stochastic differential equations (\ref{bsde-a}).

Suppose that\vspace*{2pt} $(Y,M)$ and $(\hat{Y},\hat{M})$ are two solutions
satisfying (%
\ref{i-eqa1}), where $Y$ and $\hat{Y}$ are two special
semimartingales. Let
\[
Z_{t}=M_{t}+\sum_{i=1}^{d}\int_{\tau
}^{t}f_{i}(s,Y_{s})\,dB_{s}^{i}.
\]
Then
%
\begin{equation}\label{it-0w1}
Y_{t}-\xi=\int_{t}^{T}f_{0}(s,Y_{s},\mathbf{L}(Z,Y)_{s})\,ds+Z_{t}-Z_{T}%
\end{equation}
for $t\in[ \tau,T]$, where
\[
\mathbf{L}( Z,Y) =L\Biggl( Z-\sum_{i=1}^{d}\int_{\tau
}^{\cdot}f_{i}(s,Y_{s})\,dB_{s}^{i}\Biggr) .
\]
It follows that
\[
Y_{t}=\mathbf{E}[ \xi+A_{T}|\mathcal{F}_{t}] -A_{t},
\]
where
\[
A_{t}=\int_{\tau}^{t}f_{0}(s,Y_{s},\mathbf{L}(Z,Y)_{s})\,ds\qquad
\forall t\in[ \tau,T].
\]
Hence $Y_{t}=Y(A)_{t}$ and the integral equation (\ref{it-0w1})
becomes
\[
Y_{t}=A_{T}-Z_{T}+\xi-A_{t}+Z_{t}.
\]
Since $A_{\tau}=0$ so that
\[
Y_{\tau}=A_{T}-Z_{T}+\xi+Z_{\tau},
\]
and thus we may rewrite the previous identity as
\[
Y_{t}=Y_{\tau}+( Z_{t}-Z_{\tau}) -A_{t}.
\]
By the uniqueness of the decompositions for special semimartingales
we must have
\[
Y_{\tau}+( Z_{t}-Z_{\tau}) =\mathbf{E}[ \xi+A_{T}|%
\mathcal{F}_{t}] =M(A)_{t}.
\]
Since $L(M)=L(M-M_{\tau})$ for any $M\in\mathcal{M}^{2}([\tau,T];%
\mathbf{R}^{d^{\prime}})$, so that $\mathbf{L}(Z,Y)=\mathbf{L}%
(M(A),Y)$. Hence
\[
A_{t}=\int_{\tau}^{t}f_{0}(s,Y(A)_{s},\mathbf{L}(M(A),Y(A))_{s})\,ds.
\]
The same argument applies to $(\tilde{Y},\tilde{M})$, so that we
also have
\[
\tilde{A}_{t}=\int_{\tau}^{t}f_{0}(s,Y(\tilde{A})_{s},\mathbf{L}(M(
\tilde{A}),Y(\tilde{A}))_{s})\,ds.
\]
By Theorem~\ref{lemma3}, $A=\tilde{A}$, which yields that
$Y=\tilde{Y}$. It follows then
\[
Z_{t}-Z_{\tau}=\tilde{Z}_{t}-\tilde{Z}_{\tau}\qquad\forall
t\in[ \tau,T]
\]
thus $M-M_{\tau}=\tilde{M}-\tilde{M}_{\tau}$ which
completes the proof.
\end{pf}

One of course wonders whether the global existence can be
established, by means of weighted norms, for example, as in the BSDE
literature. The present authors were unable to achieve better
results than the local existence even with different choices of
norms or spaces to which we apply the fixed point theorem. In fact,
under the Lipschitz condition only on the mapping $L$, the local
existence is the best we can hope. This is because $L(M)_{t}$ may
depend on the whole path from $\tau$ to $T$, and therefore the
corresponding stochastic functional differential equation
\[
dV_{t}=f_{0}(t,Y(V)_{t},L(M(V))_{t})\,ds+%
\sum_{i=1}^{d}f_{i}(t,Y(V)_{t})\,dB_{t}^{i},\qquad V_{\tau}=0,
\]
is neither local nor Markovian. This can be best demonstrated by its
associated differential and integral equation. For example, it is
not
difficult to show that%
\[
L_{c}(M)_{t}=\sqrt{\mathbf{E}( \langle M^{c},M^{c}\rangle
_{T}-\langle M^{c},M^{c}\rangle_{t}|\mathcal{F}_{t}) }
\]
for $t\in[ \tau,T]$ is Lipschitz continuous, where $M^{c}$
is its continuous martingale part such that $M_{0}^{c}=0$, and
therefore we have
\begin{corollary}
\label{coros1}Suppose $T\leq l_{2}$. Then there is a unique special
semimartingale $Y=(Y_{t})_{t\in[ 0,T]}$ such that $Y_{T}=\xi
$ and
%
\begin{equation}\label{bsde-r1}
Y_{t}-\xi=\int_{t}^{T}f_{0}(s,Y_{s},L_{c}(M))\,ds+M_{t}-M_{T}.
\end{equation}
Moreover $M$ is unique up to a random variable measurable with respect
to~$\mathcal{F}_{0}$.
\end{corollary}

Let us apply Corollary~\ref{coros1} to the case that $(\mathcal{F}%
_{t})_{t\geq0}$ is the Brownian filtration of Brownian motion $%
B=(B^{1},\ldots,B^{d})$. Then, by It\^{o}'s martingale
representation theorem,
\[
L_{c}(M)_{t}=\sqrt{\int_{t}^{T}\sum_{i=1}^{d}\mathbf{E}(|Z_{s}^{i}|^{2}|%
\mathcal{F}_{t})\,ds},
\]
where $Z^{i}$ are predictable processes such that
\[
M_{T}-M_{\tau}=\sum_{i=1}^{d}\int_{\tau
}^{T}Z_{t}^{i}\,dB_{t}^{i}.
\]
Suppose $u$ is a bounded, smooth function which solves the backward
parabolic nonlinear equation
%
\begin{equation}\label{diff-int}
\frac{\partial}{\partial t}u+\frac{1}{2}\Delta
u+f_{0}(t,u,K(u))=0\qquad\mbox{on }[\tau,T]\times R^{d},
\end{equation}
with $u(T,\cdot)=\varphi$, where
\[
K(u)(t,x)=\sqrt{\int_{t}^{T}P_{s-t}|\nabla u|^{2}(s,x)\,ds},
\]
where $(P_{t})_{t\geq0}$ is the heat semi-group in
$\mathbf{R}^{d}$, that is, $P_{t}=e^{{(t\Delta)}/{2}}$. In
particular, the differential and integral equation (\ref{diff-int})
is not local, and is a nonlinear equation involving space--time
integration and partial derivatives.

Applying It\^{o}'s formula to the process $Y_{t}=u(t,B_{t})$ one has
\begin{eqnarray*}
Y_{T}-Y_{t} &=&\int_{t}^{T}\biggl( \frac{\partial}{\partial t}+\frac{1}{2}%
\Delta\biggr) u(s,B_{s})\,ds+M_{T}-M_{t} \\
&=&-\int_{t}^{T}f_{0}(s,Y_{s},K(u)(s,B_{s}))\,ds+M_{T}-M_{t},
\end{eqnarray*}
where $M_{t}=\int_{0}^{t}\nabla u(s,B_{s})\,dB_{s}$ is a
square-integrable martingale, and one recognizes that
\begin{eqnarray*}
L_{c}(M)_{t} &=&\sqrt{\mathbf{E}( \langle M,M\rangle
_{T}-\langle
M,M\rangle_{t}|\mathcal{F}_{t}) } \\
&=&\sqrt{\mathbf{E}\biggl( \int_{t}^{T}|\nabla u|^{2}(s,B_{s})\,ds|\mathcal{
F}_{t}\biggr) } \\
&=&\sqrt{\int_{t}^{T}P_{s-t}|\nabla u|^{2}(s,B_{t})\,ds} \\
&=&K(u)(t,B_{t}).
\end{eqnarray*}
Therefore $(Y,M)$ is the unique solution to (\ref{bsde-r1}), and we
have a probability representation
\[
u(t,x)=\mathbf{E}\{Y_{t}|B_{t}=x\}.
\]
Since the nonlinear equation (\ref{diff-int}) depends on the
``future'' of the solution from
time $T$, it is not always possible that a solution exists back to
any time $\tau$. In turn, we thus cannot expect that the general
BSDE (\ref{bsde-a}) have a solution that is global in time without further
restrictions on $L$.

\section{Global solutions}\label{sec4}

In the previous section, under only the Lipschitz conditions on $L$
we are able to construct a solution to the backward stochastic
differential equation (\ref{bsde-a}) back to time $\tau$ such that
$T-\tau\leq l_{2}$.

In this section we construct the unique global solution to
(\ref{bsde-a}) if $L$ satisfies further regularity conditions.

We assume that the mapping $L\dvtx\mathcal{M}^{2}([0,T];\mathbf{R}%
^{d^{\prime}})\rightarrow
\mathcal{H}^{2}([0,T];\mathbf{R}^{m})$ (resp.,
$\mathcal{C}([0,T];\mathbf{R}^{m})$) satisfies three technical
conditions (a), (b) and (c) below: the local-in-time property, the
differential property and the Lipschitz condition. The last one is
standard, but the
first two properties are motivated by the example of density processes
in It%
\^{o}'s martingale representations.

For any $[T_{2},T_{1}]\subset[ 0,T]$, define the restriction
\[
L_{[T_{2},T_{1}]}\dvtx\mathcal{M}^{2}([T_{2},T_{1}];\mathbf{R}^{d^{\prime
}})\rightarrow
\mathcal{H}^{2}([T_{2},T_{1}];\mathbf{R}^{m})\qquad\mbox{(resp., }
\mathcal{C}([T_{2},T_{1}];\mathbf{R}^{d^{\prime}})\mbox{)}
\]
by $L_{[T_{2},T_{1}]}(N)_{t}=L(\hat{N})_{t}$ for any $N\in\mathcal{M}%
^{2}([T_{2},T_{1}];\mathbf{R}^{d^{\prime}})$ and $t\in[
T_{2},T_{1}]$, where $\hat{N}\in\mathcal{M}^{2}([0,T];\mathbf{R}%
^{d^{\prime}})$ defined by $\hat{N}_{t}=\mathbf{E}(N_{T_{1}}|\mathcal{F}%
_{t})$ for $t\leq T_{1}$ and $\hat{N}_{t}=N_{T_{1}}$ for $t\geq T_{1}$.%

(a) (\textit{Local-in-time property}.) For every pair of nonnegative
rational
numbers $T_{2}<T_{1}\leq T$, and for any $M\in\mathcal{M}^{2}([0,T];%
\mathbf{R}^{d^{\prime}})$, $L(M)=L_{[T_{2},T_{1}]}(\tilde{M})$ on $
(T_{2},T_{1})$, where $\tilde{M}=(M_{t})_{t\in[
T_{2},T_{1}]}$ is restriction of $M$ on $[T_{2},T_{1}]$. The
local-in-time property requires
that $L(M)_{t}$ is locally defined, that is, $L(M)_{t}$ depends only on
$%
(M_{s})_{s\in[ t,t+\varepsilon)}$ for however small the $%
\varepsilon>0$.

(b) (\textit{Differential property}.) For every pair of nonnegative
rational
numbers $T_{1}<T_{2}\leq T$ and $M\in\mathcal{M}^{2}([T_{2},T_{1}];%
\mathbf{R}^{d^{\prime}})$, one has $%
L_{[T_{2},T_{1}]}(M-M_{T_{2}})=L_{[T_{2},T_{1}]}(M)$ on
$(T_{2},T_{1})$. The differential property requires that
$L_{[T_{2},T_{1}]}(M)_{t}$ depends only on the increments
$\{M_{s}-M_{T_{2}}\dvtx s\geq t\}$ for $t\in[
T_{2},T_{1}]$.

(c) (\textit{Lipschitz continuity}.) $L\dvtx\mathcal{M}^{2}([0,T];\mathbf{R}
^{d^{\prime}})\rightarrow\mathcal{H}^{2}([0,T];\mathbf{R}^{m})$
(resp., $\mathcal{C}([0$, $T];\mathbf{R}^{m})$) is bounded and Lipschitz
continuous: there is a constant $C_{1}$ such that
%
\begin{equation} \label{a1}
\|L(M)\|_{\mathcal{H}_{[T_{2},T_{1}]}^{2}}\leq C_{1}\|M\|_{\mathcal{C}%
[T_{2},T_{1}]}
\end{equation}
and
%
\begin{equation} \label{a2}
\|L(M)-L(\tilde{M})\|_{\mathcal{H}_{[T_{2},T_{1}]}^{2}}\leq
C_{1}\|M-\tilde{M%
}\|_{\mathcal{C}[T_{2},T_{1}]}
\end{equation}
[resp.,
%
\begin{equation} \label{b1}
\|L(M)\|_{\mathcal{C}[T_{2},T_{1}]}\leq
C_{1}\|M\|_{\mathcal{C}[T_{2},T_{1}]}
\end{equation}
and
%
\begin{equation} \label{b2}
\|L(M)-L(\tilde{M})\|_{\mathcal{C}[T_{2},T_{1}]}\leq C_{1}\|M-\tilde
{M}\|_{%
\mathcal{C}[T_{2},T_{1}]}\mbox{]}
\end{equation}
for any $M,\tilde{M}\in
\mathcal{M}^{2}([0,T];\mathbf{R}^{d^{\prime}})$ and for any
rationales $T_{1}$ and $T_{2}$ such that $0\leq T_{2}<T_{1}\leq T $.
That is to say $L_{[T_{2},T_{1}]}$ are Lipschitz continuous with
Lipschitz constant independent of $[T_{2},T_{1}]\subset[
0,T]$.

The first example below provides the most interesting examples of
$L$ in applications, which are, however, variations of the classical
example considered in the literature.
\begin{Example}\label{examp1}
Suppose that $(\mathcal{F}_{t})_{t\geq0}$ is the
Brownian filtration generated by a $d+N$-dimensional Brownian motion $%
B=(B^{1},\ldots,B^{d},W^{1},\ldots,W^{N})$ on a probability space
$(\Omega
,\mathcal{F},\mathbf{P})$. If $M\in\mathcal
{M}^{2}([0,T];\mathbf{R}%
^{d^{\prime}})$, then, according to It\^{o}'s martingale
representation
theorem, $M$ is continuous, and there are unique predictable processes $
(Z_{t}^{j,i})_{t\in[ 0,T]}$ such that
%
\begin{equation} \label{m-r1}\qquad
M_{t}^{j}-M_{0}^{j}=\sum_{i=1}^{d}\int_{0}^{t}Z_{s}^{j,i}\,dB_{s}^{i}+%
\sum_{k=1}^{N}\int_{0}^{t}Z_{s}^{j,k+d}\,dW_{s}^{k},\qquad
j=1,\ldots,d^{\prime},
\end{equation}
for all $t\in[ 0,T]$. Assign $M\in\mathcal{M}^{2}([0,T];\mathbf{R%
}^{d^{\prime}})$ with $L(M)=(Z^{j,i})_{j\leq d^{\prime},i\leq d}$.
For $%
0\leq T_{2}<T_{1}\leq T$, the restriction\vspace*{1pt} of $M$ on $[T_{2},T_{1}]$,
denoted again by $M$, belongs to $\mathcal
{M}^{2}([T_{2},T_{1}];\mathbf{R}%
^{d^{\prime}})$. By the uniqueness of It\^{o}'s representation we
can see that $L$ satisfies the local-in-time and differential
properties. It is also easy to show that
$L\dvtx\mathcal{M}^{2}([0,T];\mathbf{R}^{d^{\prime}})\rightarrow
\mathcal{H}^{2}([0,T];\mathbf{R}^{d^{\prime}\times d})$
satisfies the Lipschitz condition.
\end{Example}

Another class of interesting examples of $L$ is presented in the
following example.
\begin{Example}\label{examp2}
Let $(\Omega,\mathcal{F},\mathcal
{F}_{t},\mathbf{P}%
) $ be a filtered probability space which satisfies the technical
conditions described at the beginning of Section~\ref{sec2}, but not
necessary to be a Brownian
filtration. Let $B=(B_{t})_{t\geq0}$ be a Brownian motion in\vadjust{\goodbreak}
$\mathbf{R}%
^{m}$ adapted to $(\mathcal{F}_{t})_{t\geq0}$, therefore $(\mathcal{F}%
_{t})_{t\geq0}$ is in general bigger than the Brownian filtration
generated
by $B$. Let $\mathcal{M}_{B}$ denote the closed stable sub-space of $%
\mathcal{M}_{2}$ determined by $B$, that is,
\[
\mathcal{M}_{B}=\Biggl\{ \sum_{j=1}^{m}\int_{0}^{\cdot
}Z_{s}^{j}\,dB_{s}^{j}\dvtx Z^{j}\mbox{ are predictable and }\mathbf{E}%
\int_{0}^{T}|H_{s}^{j}|^{2}\,ds<\infty\Biggr\}.
\]
Then any martingale $M$ has a unique decomposition%
\[
M_{t}-M_{0}=\sum_{j=1}^{m}\int_{0}^{t}Z_{s}^{j}\,dB_{s}^{j}+M_{t}^{\prime
},
\]
where $M^{\prime}\in\mathcal{M}^{2}([0,T];\mathbf{R})$ orthogonal
to $%
\mathcal{M}_{B}$. Then $L(M)=(Z^{j})$ satisfies the local-in-time
and differential properties, as well as the Lipschitz condition.
\end{Example}

In the following theorems we retain the basic assumptions on the
coefficients $f_{j}^{i}$ and the terminal values $\xi^{i}$.
\begin{theorem}
\label{theorem} Assume that $L$ satisfies conditions \textup{(a), (b)} and
\textup{(c)}
listed above. Then there exists a pair of processes $(Y,M)$, where $%
Y=(Y_{t})_{t\in[ 0,T]}$ is a special semimartingale, and $%
M=(M_{t})_{t\in[ 0,T]}$ is a square integrable martingale,
which solves the backward equation
%
\begin{equation} \label{bsde1m}\quad
dY_{t}=-f_{0}(t,Y_{t},L(M)_{t})\,dt-%
\sum_{i=1}^{d}f_{i}(t,Y_{t})\,dB_{t}^{i}+dM_{t},\qquad Y_{T}=\xi.
\end{equation}
The solution $Y$ is unique, and its martingale correction term $M$ is
unique up to a random variable measurable with respect to
$\mathcal{F}_{0}$.
\end{theorem}

The remainder of this section is devoted to the proof of Theorem \ref%
{theorem}.
\begin{pf*}{Proof of Theorem~\ref{theorem}}
Recall that
\[
l_{2}=\frac{1}{C_{2}^{2}[ 4C_{1}+6( 1+2\sqrt{d}) +3\sqrt{2}%
dC_{1}C_{2}] ^{2}}\wedge1,
\]
which is positive and independent of $\xi$.

By Theorem~\ref{lemma3}, if the terminal time $T\leq l_{2}$, the
nonlinear mapping $\mathbb{L}$ on
$\mathcal{C}_{0}([0,T];\mathbf{R}^{d^{\prime}})$
admits a unique fixed point, where%
\[
\mathbb{L}(V)_{t}=\int_{0}^{t}f_{0}(s,Y(V)_{s},L(M(V))_{s})\,ds+\sum
_{i=1}^{d}%
\int_{0}^{t}f_{i}(s,Y(V)_{s})\,dB_{s}^{i}.
\]

Next we consider the case $T>l_{2}$. In this case we divide the
interval $%
[0,T]$ into subintervals with length not exceeding $l_{2}$. More
precisely, let
\[
T=T_{0}>T_{1}>\cdots>T_{k}=0
\]
so that $0<T_{i-1}-T_{i}\leq l_{2}$ where $T_{i}$ are rationales except\vadjust{\goodbreak}
$%
T_{0}=T$.

Begin with the top interval $[T_{1},T_{0}]$, together with the
terminal
value $Y_{T_{0}}=\xi$ and the filtration starting from $\mathcal
{F}_{T_{1}}$. Applying Lemma~\ref{lemma3} to the interval $[T_{1},T_{0}]$ and
$\mathbb{L}%
_{1}$, where
\begin{eqnarray*}
(\mathbb{L}_{1}V)_{t}
&=&\int
_{T_{1}}^{t}f_{0}\bigl(s,Y_{1}(V)_{s},L_{[T_{1},T_{0}]}(M_{1}(V))_{s}\bigr)\,ds \\
&&{}+\sum_{i=1}^{d}\int_{T_{1}}^{t}f_{i}(s,Y_{1}(V)_{s})\,dB_{s}^{i},
\end{eqnarray*}
where
\[
M_{1}(V)_{t}=\mathbf{E}( \xi
+V_{T_{0}}|\mathcal{F}_{t}) ,\qquad
Y_{1}(V)_{t}=M_{1}(V)_{t}-V_{t}
\]
for any $V\in\mathcal{C}([T_{1},T_{0}];\mathbf{R}^{d^{\prime}})$
and $%
t\in[ T_{1},T_{0}]$. Then, there exists a unique $V(1)\in\mathcal{C}%
_{0}([T_{1},T_{0}];\mathbf{R}^{d^{\prime}})$ such that $\mathbb{L}%
_{1}V(1)=V(1)$.

Repeat the same argument to each interval $[T_{j},T_{j-1}]$ (for
$2\leq j\leq k$) with the terminal value
$Y_{j-1}(V(j-1))_{T_{j-1}}$, the
filtration starting from $\mathcal{F}_{T_{j}}$, and the nonlinear
mapping $%
\mathbb{L}_{j}$ defined on $\mathcal{C}_{0}([T_{j},T_{j-1}];\mathbf{R}%
^{d^{\prime}})$ by
\begin{eqnarray*}
(\mathbb{L}_{j}V)_{t}
&=&\int_{T_{j}}^{t}f_{0}\bigl(s,Y_{j}(V)_{s},L_{[T_{j},T_{j-1}]}(M(V_{j}))_{s}\bigr)\,ds
\\
&&{}+\sum_{i=1}^{N}\int_{T_{j}}^{t}f_{i}(s,Y_{j}(V)_{s})\,dB_{s}^{i},
\end{eqnarray*}
where $V\in\mathcal{C}([T_{j},T_{j-1}];\mathbf{R}^{d^{\prime
}})$ and
\begin{eqnarray*}
M_{j}(V)_{t}& = &\mathbf{E}\bigl(Y_{j-1}\bigl(V(j-1)\bigr)_{T_{j-1}}+V_{T_{j-1}}|\mathcal{%
F}_{t}\bigr), \\
Y_{j}(V)_{t}& = &M_{j}(V)_{t}-V_{t}
\end{eqnarray*}
for $t\in[ T_{j},T_{j-1}]$.

Therefore, for $1\leq j\leq k$, there exists a unique $V(j)\in
C([T_{j},T_{j-1}];\mathbf{R}^{d^{\prime}})$ such that
\begin{eqnarray*}
V(j)_{t}
&=&\int_{T_{j}}^{t}f_{0}\bigl(s,Y(j)_{s},L_{[T_{j},T_{j-1}]}(M(j))_{s}\bigr)\,ds
\\
&&{}+\sum_{i=1}^{N}\int_{T_{j}}^{t}f_{i}(s,Y(j)_{s})\,dB_{s}^{i}
\end{eqnarray*}
for $t\in[ T_{j},T_{j-1}]$, where $Y(0)_{T_{0}}=\xi$, $%
Y(j-1)_{T_{j-1}}=Y(j)_{T_{j-1}}$ for $2\leq j\leq k$, and
\begin{eqnarray*}
M(j)_{t}& = &\mathbf{E}\bigl(Y(j-1)_{T_{j-1}}+V(j)_{T_{j-1}}|\mathcal
{F}_{t}\bigr)%
, \\
Y(j)_{t}& = &M(j)_{t}-V(j)_{t}
\end{eqnarray*}
for $t\in[ T_{j},T_{j-1}]$.\vadjust{\goodbreak}

Since $Y(j-1)_{T_{j-1}}=Y(j)_{T_{j-1}}$ for $2\leq j\leq k$, $%
Y=(Y_{t})_{t\in[ 0,T]}$ given by
\[
Y_{t}=Y(j)_{t}\qquad\mbox{if }t\in[ T_{j},T_{j-1}]
\]
for $1\leq j\leq k$, is well defined. Define $V$ by shifting it at
the partition points,
\[
V_{t}=\cases{
V(k)_{t}, &\quad if $t\in[ 0,T_{k-1}]$, \cr
V(k-1)_{t}+V(k)_{T_{k-1}}, &\quad if $t\in[ T_{k-1},T_{k-2}]$, \cr
\cdots& \cr
\displaystyle V(1)_{t}+\sum_{l=2}^{k}V(l)_{T_{l-1}}, &\quad if $t\in[ T_{1},T]$.}
\]
Then $V\in\mathcal{C}([0,T];\mathbf{R}^{d^{\prime}})$. Finally
we define
\[
M_{t}=Y_{t}+V_{t}\qquad\mbox{for }t\in[ 0,T].
\]
It remains to show that $M$ is a martingale.
\begin{lemma}
\label{le4} $M$ defined above has the expression
%
\begin{equation} \label{ex1}
M_{t}=M(j)_{t}+\sum_{l=j+1}^{k}V(l)_{T_{l-1}} \qquad\mbox{if } t\in
[ T_{j},T_{j-1}]
\end{equation}
for $1\leq j\leq k$, and moreover, $M$ is an
$(\mathcal{F}_{t})$-martingale up to time $T$, so that
\[
M_{t}=\mathbf{E}(\xi+V_{T}|\mathcal{F}_{t}).
\]
\end{lemma}
\begin{pf}
We first prove the expression (\ref{ex1}). Since for $1\leq j \leq
k$,
\[
Y(j)_{t}=M(j)_{t}-V(j)_{t}\qquad\mbox{if }t\in[
T_{j},T_{j-1}]
\]
so that
\[
Y_{t}=M(j)_{t}+\sum_{l=j+1}^{k}V(l)_{T_{l-1}}-V_{t}\qquad\mbox{if
}t\in[ T_{j},T_{j-1}],
\]
one may conclude that
\[
M_{t}=M(j)_{t}+\sum_{l=j+1}^{k}V(l)_{T_{l-1}}\qquad\mbox{if }t\in
[ T_{j},T_{j-1}].
\]

It is clear that $M$ is adapted to $(\mathcal{F}_{t})$, so we only
need to show $\mathbf{E}(M_{t}|\mathcal{F}_{s})=M_{s}$ for any
$0\leq s\leq t\leq T$. If $s,t\in[ T_{j},T_{j-1}]$ for some
$j$, then
\[
M_{t}-M_{s}=M(j)_{t}-M(j)_{s}
\]
so that
\[
\mathbf{E}(M_{t}-M_{s}|\mathcal{F}_{s})=\mathbf{E}\bigl(M(j)_{t}-M(j)_{s}|%
\mathcal{F}_{s}\bigr)=0 .
\]
If $s\in[ T_{i},T_{i-1}]$ and $t\in[ T_{j},T_{j-1}]$ for some $%
i>j$, then according to (\ref{ex1}),
\[
M_{s}=M(i)_{s}+\sum_{l=i+1}^{k}V(l)_{T_{l-1}}\vadjust{\goodbreak}
\]
and
\[
M_{t}=M(j)_{t}+\sum_{l=j+1}^{k}V(l)_{T_{l-1}}.
\]
Since $M(j)$ is a martingale on $[T_{j},T_{j-1}]$ so that
\[
\mathbf{E}(M_{t}|\mathcal{F}_{T_{j}})=M(j)_{T_{j}}+%
\sum_{l=j+1}^{k}V(l)_{T_{l-1}},
\]
conditional on $\mathcal{F}_{T_{j+1}}\subset\mathcal{F}_{T_{j}}$
we obtain
%
\begin{equation} \label{jq1}\quad
\mathbf{E}(M_{t}|\mathcal{F}_{T_{j+1}})=\mathbf{E}%
\bigl(M(j)_{T_{j}}+V(j+1)_{T_{j}}|\mathcal{F}_{T_{j+1}}\bigr)+%
\sum_{l=j+2}^{k}V(l)_{T_{l-1}}.
\end{equation}
On the other hand, $M(j)_{T_{j}}=Y_{T_{j}}+V(j)_{T_{j}}=Y_{T_{j}}$
so that
\begin{eqnarray*}
\mathbf{E}\bigl(M(j)_{T_{j}}+V(j+1)_{T_{j}}|\mathcal{F}_{T_{j+1}}\bigr) &=&%
\mathbf{E}\bigl(Y_{T_{j}}+V(j+1)_{T_{j}}|\mathcal{F}_{T_{j+1}}\bigr) \\
&=&M(j+1)_{T_{j+1}}.
\end{eqnarray*}
Substituting it into (\ref{jq1}) we obtain
%
\begin{equation} \label{jq2}
\mathbf{E}(M_{t}|\mathcal{F}_{T_{j+1}})=M(j+1)_{T_{j+1}}+%
\sum_{l=j+2}^{k}V(l)_{T_{l-1}}.
\end{equation}
By repeating the same argument we may establish
%
\begin{equation} \label{jq3}
\mathbf{E}(M_{t}|\mathcal{F}_{T_{i-1}})=M(i-1)_{T_{i-1}}+%
\sum_{l=i}^{k}V(l)_{T_{l-1}}.
\end{equation}
Since $s\in[ T_{i},T_{i-1}]$, conditional on
$\mathcal{F}_{s}$,
\begin{eqnarray*}
\mathbf{E}(M_{t}|\mathcal{F}_{s})& = & \mathbf{E}%
\bigl(M(i-1)_{T_{i-1}}+V(i)_{T_{i-1}}|\mathcal{F}_{s}\bigr)+%
\sum_{l=i+1}^{k}V(l)_{T_{l-1}} \\
& = &\mathbf{E}\bigl(Y_{T_{i-1}}+V(i)_{T_{i-1}}|\mathcal{F}%
_{s}\bigr)+\sum_{l=i+1}^{k}V(l)_{T_{l-1}} \\
& = &M(i)_{s}+\sum_{l=i+1}^{k}V(l)_{T_{l-1}} \\
& = &M_{s},
\end{eqnarray*}
which proves $M$ is an $\mathcal{F}_{t}$-adapted martingale up to
$T$.
\end{pf}

Since $L$ satisfies the local-in-time property and the differential
property, so that
\[
L_{[T_{j},T_{j-1}]}(M(V_{j}))_{s}=L(M)_{s}\qquad\mbox{for }s\in[
T_{j},T_{j-1}],
\]
hence
\[
V(j)_{t}=\int_{T_{j}}^{t}f_{0}(s,Y_{s},L(M)_{s})\,ds+\sum_{i=1}^{d}%
\int_{T_{j}}^{t}f_{i}(s,Y_{s})\,dB_{s}^{i}
\]
for any $t\in[ T_{j},T_{j-1}]$ and $j=2,\ldots,k$. Therefore
\[
V_{t}=\int_{0}^{t}f_{0}(s,Y_{s},L(M)_{s})\,ds+\sum_{i=1}^{d}%
\int_{0}^{t}f_{i}(s,Y_{s})\,dB_{s}^{i}\qquad\forall t\in
[ 0,T]
\]
and $Y=M-V$, $Y_{T}=\xi$, which together imply that
\[
M_{t}-Y_{t}=\int_{0}^{t}f_{0}(s,Y_{s},L(M)_{s})\,ds+\sum_{i=1}^{d}%
\int_{0}^{t}f_{i}(s,Y_{s})\,dB_{s}^{i}\qquad\forall t\in[ 0,T]%
.
\]
Thus $(Y,M)$ solves the backward equation (\ref{bsde-a}). Uniqueness
follows from the fact the solution $(Y(j),M(j)-M(j)_{T_{j}})$ is
unique for any $j$.

The proof of Theorem~\ref{theorem} is complete.
\end{pf*}

We end this article with several comments about the main results.

The local and global existence results remain valid even if the driver $
f_{0}^{j}$ and the diffusion coefficients $f_{i}^{j}$ of the BSDE
are random
as long as the global Lipschitz conditions are maintained. For example,
if $%
f_{0}^{j}\dvtx\mathbf{R}_{+}\times\Omega\times
\mathbf{R}^{d^{\prime
}}\times\mathbf{R}^{m}\rightarrow\mathbf{R}^{d^{\prime}}$
and $%
f_{i}^{j}\dvtx\mathbf{R}_{+}\times\Omega\times
\mathbf{R}^{d^{\prime}}\rightarrow\mathbf{R}^{d^{\prime
}}$ are jointly measurable such that
for any special semimartingale $Y$ and $Z\in\mathcal{H}^{2}([0,T];%
\mathbf{R}^{m})$ (resp., $\mathcal{C}([0,T];\mathbf{R}^{m})$), $%
f_{0}^{j}(t,\cdot,Y_{t},Z_{t})$ and $f_{i}^{j}(t,\cdot,Y_{t})$ are
progressively measurable and
\begin{eqnarray*}
&&\sqrt{\mathbf{E}\int_{T_{2}}^{T_{1}}|f_{0}^{j}(t,\cdot
,Y_{t},Z_{t})-f_{0}^{j}(t,\cdot,\tilde{Y}_{t},\tilde{Z}_{t})|^{2}\,dt} \\
&&\qquad\leq C_{3}\|Y-Y\|_{\mathcal{C}[T_{2},T_{1}]}+C_{3}\|Z-Z\|_{\mathcal{H}
^{2}[T_{2},T_{1}]}
\end{eqnarray*}
and%
\[
\sqrt{\mathbf{E}\int_{T_{2}}^{T_{1}}|f_{i}^{j}(t,\cdot
,Y_{t})-f_{i}^{j}(t,\cdot,\tilde{Y}_{t})|^{2}\,dt}\leq
C_{3}\|Y-Y\|_{\mathcal{%
C}[T_{2},T_{1}]}
\]
for any $[T_{2},T_{1}]\subset[ 0,T]$, $Y,\tilde{Y}\in\mathcal{S}%
([0,T];\mathbf{R}^{d})$ and $Z$, $\tilde{Z}\in\mathcal
{H}^{2}([0,T];%
\mathbf{R}^{m})$ (and similarly for the case $Z$, $\tilde{Z}\in
\mathcal{C%
}([0,T];\mathbf{R}^{m})$ with norm $\mathcal{C}[T_{2},T_{1}]$
instead of $\mathcal{H}^{2}[T_{2},T_{1}]$), then all our local and
global results remain true. We leave the details of the proofs for
the reader who may be interested in such a generalization.

\section*{Acknowledgments}
The authors wish to thank Professor Yves LeJan and the referee for
their comments and suggestions on the presentation.


%
\printaddresses


\begin{thebibliography}{42}

\bibitem{MR1233625}
\begin{barticle}[mr]
\bauthor{\bsnm{Antonelli},~\bfnm{Fabio}\binits{F.}}
(\byear{1993}).
\btitle{Backward--forward stochastic differential equations}.
\bjournal{Ann. Appl. Probab.}
\bvolume{3}
\bpages{777--793}.
\bid{mr={1233625}}
\end{barticle}
\endbibitem

\bibitem{MR1978231}
\begin{barticle}[vtex]
\bauthor{\bsnm{Antonelli},~\bfnm{Fabio}\binits{F.}} \AND
  \bauthor{\bsnm{Ma},~\bfnm{Jin}\binits{J.}}
(\byear{2003}).
\btitle{Weak solutions of forward--backward {SDE}'s}.
\bjournal{Stoch. Anal. Appl.}
\bvolume{21}
\bpages{493--514}.
\bid{doi={10.1081/SAP-120020423}, mr={1978231}}
\end{barticle}
\endbibitem

\bibitem{MR2127730}
\begin{barticle}[mr]
\bauthor{\bsnm{Bally},~\bfnm{V.}\binits{V.}},
  \bauthor{\bsnm{Pardoux},~\bfnm{E.}\binits{E.}} \AND
  \bauthor{\bsnm{Stoica},~\bfnm{L.}\binits{L.}}
(\byear{2005}).
\btitle{Backward stochastic differential equations associated to a symmetric
  {M}arkov process}.
\bjournal{Potential Anal.}
\bvolume{22}
\bpages{17--60}.
\bid{mr={2127730}}
\end{barticle}
\endbibitem

\bibitem{MR1436432}
\begin{barticle}[mr]
\bauthor{\bsnm{Barles},~\bfnm{Guy}\binits{G.}},
  \bauthor{\bsnm{Buckdahn},~\bfnm{Rainer}\binits{R.}} \AND
  \bauthor{\bsnm{Pardoux},~\bfnm{Etienne}\binits{E.}}
(\byear{1997}).
\btitle{Backward stochastic differential equations and integral-partial
  differential equations}.
\bjournal{Stochastics Stochastics Rep.}
\bvolume{60}
\bpages{57--83}.
\bid{mr={1436432}}
\end{barticle}
\endbibitem

\bibitem{Bismut}
\begin{bmisc}[vtex]
\bauthor{\bsnm{Bismut},~\bfnm{J.~M.}\binits{J.~M.}}
(\byear{1973}).
\bhowpublished{Analyse convexe et probabiliti\'es. These, facult\'e des sciences
de Paris, Paris}.
\end{bmisc}
\endbibitem

\bibitem{MR0453161}
\begin{barticle}[mr]
\bauthor{\bsnm{Bismut},~\bfnm{Jean-Michel}\binits{J.-M.}}
(\byear{1976}).
\btitle{Th\'eorie probabiliste du contr\^ole des diffusions}.
\bjournal{Mem. Amer. Math. Soc.}
\bvolume{4}
\bpages{1--130}.
\bid{mr={0453161}}
\end{barticle}
\endbibitem

\bibitem{MR0469466}
\begin{barticle}[mr]
\bauthor{\bsnm{Bismut},~\bfnm{Jean-Michel}\binits{J.-M.}}
(\byear{1978}).
\btitle{An introductory approach to duality in optimal stochastic control}.
\bjournal{SIAM Rev.}
\bvolume{20}
\bpages{62--78}.
\bid{mr={0469466}}
\end{barticle}
\endbibitem

\bibitem{Briand-re1}
\begin{barticle}[mr]
\bauthor{\bsnm{Briand},~\bfnm{Ph.}\binits{P.}},
  \bauthor{\bsnm{Delyon},~\bfnm{B.}\binits{B.}},
  \bauthor{\bsnm{Hu},~\bfnm{Y.}\binits{Y.}},
  \bauthor{\bsnm{Pardoux},~\bfnm{E.}\binits{E.}} \AND
  \bauthor{\bsnm{Stoica},~\bfnm{L.}\binits{L.}}
(\byear{2003}).
\btitle{{$L\sp p$} solutions of backward stochastic differential equations}.
\bjournal{Stochastic Process. Appl.}
\bvolume{108}
\bpages{109--129}.
\bid{doi={10.1016/S0304-4149(03)00089-9}, mr={2008603}}
\end{barticle}
\endbibitem

\bibitem{MR2257138}
\begin{barticle}[mr]
\bauthor{\bsnm{Briand},~\bfnm{Philippe}\binits{P.}} \AND
  \bauthor{\bsnm{Hu},~\bfnm{Ying}\binits{Y.}}
(\byear{2006}).
\btitle{B{SDE} with quadratic growth and unbounded terminal value}.
\bjournal{Probab. Theory Related Fields}
\bvolume{136}
\bpages{604--618}.
\bid{doi={10.1007/s00440-006-0497-0}, mr={2257138}}
\end{barticle}
\endbibitem

\bibitem{MR2391164}
\begin{barticle}[mr]
\bauthor{\bsnm{Briand},~\bfnm{Philippe}\binits{P.}} \AND
  \bauthor{\bsnm{Hu},~\bfnm{Ying}\binits{Y.}}
(\byear{2008}).
\btitle{Quadratic {BSDE}s with convex generators and unbounded terminal
  conditions}.
\bjournal{Probab. Theory Related Fields}
\bvolume{141}
\bpages{543--567}.
\bid{doi={10.1007/s00440-007-0093-y}, mr={2391164}}
\end{barticle}
\endbibitem

\bibitem{MR2354579}
\begin{barticle}[mr]
\bauthor{\bsnm{Buckdahn},~\bfnm{R.}\binits{R.}} \AND
  \bauthor{\bsnm{Engelbert},~\bfnm{H.~J.}\binits{H.~J.}}
(\byear{2007}).
\btitle{On the continuity of weak solutions of backward stochastic differential
  equations}.
\bjournal{Teor. Veroyatn. Primen.}
\bvolume{52}
\bpages{190--199}.
\bid{doi={10.1137/S0040585X9798292X}, mr={2354579}}%
\end{barticle}%
\endbibitem%

\bibitem{MR2141331}
\begin{barticle}[mr]
\bauthor{\bsnm{Buckdahn},~\bfnm{R.}\binits{R.}},
  \bauthor{\bsnm{Engelbert},~\bfnm{H.~J.}\binits{H.~J.}} \AND
  \bauthor{\bsnm{R{\u{a}}{\c{s}}canu},~\bfnm{A.}\binits{A.}}
(\byear{2004}).
\btitle{On weak solutions of backward stochastic differential equations}.
\bjournal{Teor. Veroyatn. Primen.}
\bvolume{49}
\bpages{70--108}.
\bid{doi={10.1137/S0040585X97980877}, mr={2141331}}%
\end{barticle}%
\endbibitem%

\bibitem{Touzi}
\begin{barticle}[vtex]
\bauthor{\bsnm{Cheridito},~\bfnm{P.}\binits{P.}},
  \bauthor{\bsnm{Soner},~\bfnm{M.}\binits{M.}},
  \bauthor{\bsnm{Touzi},~\bfnm{N.}\binits{N.}} \AND
  \bauthor{\bsnm{Victoir},~\bfnm{N.}\binits{N.}}
(\byear{2007}).
\btitle{Second order backward stochastic differential equations and fully non-linear parabolic
PDEs}.
\bjournal{Comm. Pure Appl. Math.}
\bvolume{60}
\bpages{1081--1110}.
\end{barticle}
\endbibitem

\bibitem{MR1415239}
\begin{barticle}[mr]
\bauthor{\bsnm{Cvitani{\'c}},~\bfnm{Jak{\v{s}}a}\binits{J.}} \AND
  \bauthor{\bsnm{Karatzas},~\bfnm{Ioannis}\binits{I.}}
(\byear{1996}).
\btitle{Backward stochastic differential equations with reflection and {D}ynkin
  games}.
\bjournal{Ann. Probab.}
\bvolume{24}
\bpages{2024--2056}.
\bid{doi={10.1214/aop/1041903216}, mr={1415239}}
\end{barticle}
\endbibitem

\bibitem{Duffie}
\begin{barticle}[mr]
\bauthor{\bsnm{Duffie},~\bfnm{Darrell}\binits{D.}} \AND
  \bauthor{\bsnm{Epstein},~\bfnm{Larry~G.}\binits{L.~G.}}
(\byear{1992}).
\btitle{Stochastic differential utility. With an appendix by the authors and C. Skiadas}.
\bjournal{Econometrica}
\bvolume{60}
\bpages{353--394}.
\bid{doi={10.2307/2951600}, mr={1162620}}
\end{barticle}
\endbibitem

\bibitem{ElKaroui}
\begin{bbook}[vtex]
\bauthor{\bsnm{El~Karoui},~\bfnm{N.}\binits{N.}},
  \bauthor{\bsnm{Hamadene},~\bfnm{S.}\binits{S.}} \AND
  \bauthor{\bsnm{Matoussi},~\bfnm{A.}\binits{A.}}
(\byear{2009}).
\btitle{BSDEs and Applications. Indifference Pricing: Theory and
  Applications}
\bpages{267--320}.
\bpublisher{Princeton Univ. Press}, \baddress{Princeton, NJ}.
\end{bbook}
\endbibitem

\bibitem{MR1434123}
\begin{barticle}[mr]
\bauthor{\bsnm{El~Karoui},~\bfnm{N.}\binits{N.}},
  \bauthor{\bsnm{Kapoudjian},~\bfnm{C.}\binits{C.}},
  \bauthor{\bsnm{Pardoux},~\bfnm{E.}\binits{E.}},
  \bauthor{\bsnm{Peng},~\bfnm{S.}\binits{S.}} \AND
  \bauthor{\bsnm{Quenez},~\bfnm{M.~C.}\binits{M.~C.}}
(\byear{1997}).
\btitle{Reflected solutions of backward {SDE}'s, and related obstacle problems
  for {PDE}'s}.
\bjournal{Ann. Probab.}
\bvolume{25}
\bpages{702--737}.
\bid{doi={10.1214/aop/1024404416}, mr={1434123}}
\end{barticle}
\endbibitem

\bibitem{MR1752671}
\begin{bbook}[vtex]
\beditor{\bsnm{El~Karoui},~\bfnm{Nicole}\binits{N.}} \AND
  \beditor{\bsnm{Mazliak},~\bfnm{Laurent}\binits{L.}}, eds.
(\byear{1997}).
\btitle{Backward Stochastic Differential Equations}.
\bseries{Pitman Research Notes in Mathematics Series}
\bvolume{364}.
\bpublisher{Longman}, \baddress{Harlow}.
\bid{mr={1752671}}
\end{bbook}
\endbibitem


\bibitem{MR1434407}
\begin{barticle}[mr]
\bauthor{\bsnm{El~Karoui},~\bfnm{N.}\binits{N.}},
  \bauthor{\bsnm{Peng},~\bfnm{S.}\binits{S.}} \AND
  \bauthor{\bsnm{Quenez},~\bfnm{M.~C.}\binits{M.~C.}}
(\byear{1997}).
\btitle{Backward stochastic differential equations in finance}.
\bjournal{Math. Finance}
\bvolume{7}
\bpages{1--71}.
\bid{doi={10.1111/1467-9965.00022}, mr={1434407}}
\end{barticle}
\endbibitem

\bibitem{MR1814364}
\begin{bbook}[vtex]
\bauthor{\bsnm{Gilbarg},~\bfnm{David}\binits{D.}} \AND
  \bauthor{\bsnm{Trudinger},~\bfnm{Neil~S.}\binits{N.~S.}}
(\byear{2001}).
\btitle{Elliptic Partial Differential Equations of Second Order}.
\bpublisher{Springer}, \baddress{Berlin}.
\bid{mr={1814364}}
\end{bbook}
\endbibitem

\bibitem{MR1752681}
\begin{bincollection}[vtex]
\bauthor{\bsnm{Hamadene},~\bfnm{S.}\binits{S.}},
  \bauthor{\bsnm{Lepeltier},~\bfnm{J.~P.}\binits{J.~P.}} \AND
  \bauthor{\bsnm{Matoussi},~\bfnm{A.}\binits{A.}}
(\byear{1997}).
\btitle{Double barrier backward {SDE}s with continuous coefficient}.
In \bbooktitle{Backward Stochastic Differential Equations ({P}aris,
  1995--1996)}.
\bseries{Pitman Research Notes in Mathematics Series}
\bvolume{364}
\bpages{161--175}.
\bpublisher{Longman}, \baddress{Harlow}.
\bid{mr={1752681}}
\end{bincollection}
\endbibitem

\bibitem{MR2152241}
\begin{barticle}[mr]
\bauthor{\bsnm{Hu},~\bfnm{Ying}\binits{Y.}},
  \bauthor{\bsnm{Imkeller},~\bfnm{Peter}\binits{P.}} \AND
  \bauthor{\bsnm{M{\"u}ller},~\bfnm{Matthias}\binits{M.}}
(\byear{2005}).
\btitle{Utility maximization in incomplete markets}.
\bjournal{Ann. Appl. Probab.}
\bvolume{15}
\bpages{1691--1712}.
\bid{doi={10.1214/105051605000000188}, mr={2152241}}
\end{barticle}
\endbibitem

\bibitem{MR1355060}
\begin{barticle}[mr]
\bauthor{\bsnm{Hu},~\bfnm{Y.}\binits{Y.}} \AND
  \bauthor{\bsnm{Peng},~\bfnm{S.}\binits{S.}}
(\byear{1995}).
\btitle{Solution of forward--backward stochastic differential equations}.
\bjournal{Probab. Theory Related Fields}
\bvolume{103}
\bpages{273--283}.
\bid{doi={10.1007/BF01204218}, mr={1355060}}
\end{barticle}
\endbibitem

\bibitem{MR1782267}
\begin{barticle}[mr]
\bauthor{\bsnm{Kobylanski},~\bfnm{Magdalena}\binits{M.}}
(\byear{2000}).
\btitle{Backward stochastic differential equations and partial differential
  equations with quadratic growth}.
\bjournal{Ann. Probab.}
\bvolume{28}
\bpages{558--602}.
\bid{doi={10.1214/aop/1019160253}, mr={1782267}}
\end{barticle}
\endbibitem

\bibitem{MR1766421}
\begin{barticle}[mr]
\bauthor{\bsnm{Kohlmann},~\bfnm{Michael}\binits{M.}} \AND
  \bauthor{\bsnm{Zhou},~\bfnm{Xun~Yu}\binits{X.~Y.}}
(\byear{2000}).
\btitle{Relationship between backward stochastic differential equations and
  stochastic controls: A linear-quadratic approach}.
\bjournal{SIAM J. Control Optim.}
\bvolume{38}
\bpages{1392--1407 (electronic)}.
\bid{doi={10.1137/S036301299834973X}, mr={1766421}}
\end{barticle}
\endbibitem

\bibitem{MR1602231}
\begin{barticle}[mr]
\bauthor{\bsnm{Lepeltier},~\bfnm{J.~P.}\binits{J.~P.}} \AND
  \bauthor{\bsnm{San~Martin},~\bfnm{J.}\binits{J.}}
(\byear{1997}).
\btitle{Backward stochastic differential equations with continuous
  coefficient}.
\bjournal{Statist. Probab. Lett.}
\bvolume{32}
\bpages{425--430}.
\bid{doi={10.1016/S0167-7152(96)00103-4}, mr={1602231}}
\end{barticle}
\endbibitem

\bibitem{MR1262970}
\begin{barticle}[mr]
\bauthor{\bsnm{Ma},~\bfnm{Jin}\binits{J.}},
  \bauthor{\bsnm{Protter},~\bfnm{Philip}\binits{P.}} \AND
  \bauthor{\bsnm{Yong},~\bfnm{Jiong~Min}\binits{J.~M.}}
(\byear{1994}).
\btitle{Solving forward--backward stochastic differential equations
  explicitly---a four step scheme}.
\bjournal{Probab. Theory Related Fields}
\bvolume{98}
\bpages{339--359}.
\bid{doi={10.1007/BF01192258}, mr={1262970}}
\end{barticle}
\endbibitem

\bibitem{MR1704232}
\begin{bbook}[mr]
\bauthor{\bsnm{Ma},~\bfnm{Jin}\binits{J.}} \AND
  \bauthor{\bsnm{Yong},~\bfnm{Jiongmin}\binits{J.}}
(\byear{1999}).
\btitle{Forward--Backward Stochastic Differential Equations and Their
  Applications}.
\bseries{Lecture Notes in Math.}
\bvolume{1702}.
\bpublisher{Springer}, \baddress{Berlin}.
\bid{mr={1704232}}
\end{bbook}
\endbibitem

\bibitem{MR2478677}
\begin{barticle}[mr]
\bauthor{\bsnm{Ma},~\bfnm{Jin}\binits{J.}},
  \bauthor{\bsnm{Zhang},~\bfnm{Jianfeng}\binits{J.}} \AND
  \bauthor{\bsnm{Zheng},~\bfnm{Ziyu}\binits{Z.}}
(\byear{2008}).
\btitle{Weak solutions for forward--backward {SDE}s---a martingale problem
  approach}.
\bjournal{Ann. Probab.}
\bvolume{36}
\bpages{2092--2125}.
\bid{doi={10.1214/08-AOP0383}, mr={2478677}}
\end{barticle}
\endbibitem

\bibitem{majda1}
\begin{bbook}[mr]
\bauthor{\bsnm{Majda},~\bfnm{A.}\binits{A.}}
(\byear{1984}).
\btitle{Compressible Fluid Flow and Systems of Conservation Laws in Several
  Space Variables}.
\bseries{Applied Mathematical Sciences}
\bvolume{53}.
\bpublisher{Springer}, \baddress{New York}.
\bid{mr={0748308}}
\end{bbook}
\endbibitem

\bibitem{MR1176785}
\begin{bincollection}[vtex]
\bauthor{\bsnm{Pardoux},~\bfnm{{\'E}.}\binits{{\'E}.}} \AND
  \bauthor{\bsnm{Peng},~\bfnm{S.}\binits{S.}}
(\byear{1992}).
\btitle{Backward stochastic differential equations and quasilinear parabolic
  partial differential equations}.
In \bbooktitle{Stochastic Partial Differential Equations and Their Applications
  ({C}harlotte, {NC}, 1991)}.
\bseries{Lecture Notes in Control and Inform. Sci.}
\bvolume{176}
\bpages{200--217}.
\bpublisher{Springer}, \baddress{Berlin}.
\bid{doi={10.1007/BFb0007334}, mr={1176785}}
\end{bincollection}
\endbibitem

\bibitem{MR1037747}
\begin{barticle}[mr]
\bauthor{\bsnm{Pardoux},~\bfnm{{\'E}.}\binits{{\'E}.}} \AND
  \bauthor{\bsnm{Peng},~\bfnm{S.~G.}\binits{S.~G.}}
(\byear{1990}).
\btitle{Adapted solution of a backward stochastic differential equation}.
\bjournal{Systems Control Lett.}
\bvolume{14}
\bpages{55--61}.
\bid{doi={10.1016/0167-6911(90)90082-6}, mr={1037747}}
\end{barticle}
\endbibitem

\bibitem{MR1675098}
\begin{barticle}[mr]
\bauthor{\bsnm{Peng},~\bfnm{Shige}\binits{S.}} \AND
  \bauthor{\bsnm{Wu},~\bfnm{Zhen}\binits{Z.}}
(\byear{1999}).
\btitle{Fully coupled forward--backward stochastic differential equations and
  applications to optimal control}.
\bjournal{SIAM J. Control Optim.}
\bvolume{37}
\bpages{825--843}.
\bid{doi={10.1137/S0363012996313549}, mr={1675098}}
\end{barticle}
\endbibitem

\bibitem{MR105163}
\begin{barticle}[mr]
\bauthor{\bsnm{Peng},~\bfnm{Shi~Ge}\binits{S.~G.}}
(\byear{1990}).
\btitle{A general stochastic maximum principle for optimal control problems}.
\bjournal{SIAM J. Control Optim.}
\bvolume{28}
\bpages{966--979}.
\bid{doi={10.1137/0328054}, mr={1051633}}
\end{barticle}
\endbibitem

\bibitem{MR1149116}
\begin{barticle}[mr]
\bauthor{\bsnm{Peng},~\bfnm{Shi~Ge}\binits{S.~G.}}
(\byear{1991}).
\btitle{Probabilistic interpretation for systems of quasilinear parabolic
  partial differential equations}.
\bjournal{Stochastics Stochastics Rep.}
\bvolume{37}
\bpages{61--74}.
\bid{mr={1149116}}
\end{barticle}
\endbibitem

\bibitem{MR1440399}
\begin{barticle}[mr]
\bauthor{\bsnm{Rong},~\bfnm{Situ}\binits{S.}}
(\byear{1997}).
\btitle{On solutions of backward stochastic differential equations with jumps
  and applications}.
\bjournal{Stochastic Process. Appl.}
\bvolume{66}
\bpages{209--236}.
\bid{doi={10.1016/S0304-4149(96)00120-2}, mr={1440399}}
\end{barticle}
\endbibitem

\bibitem{MR1802922}
\begin{barticle}[vtex]
\bauthor{\bsnm{Rouge},~\bfnm{Richard}\binits{R.}} \AND
  \bauthor{\bsnm{El~Karoui},~\bfnm{Nicole}\binits{N.}}
(\byear{2000}).
\btitle{Pricing via utility maximization and entropy}.
\bjournal{Math. Finance}
\bvolume{10}
\bpages{259--276}.
\bid{doi={10.1111/1467-9965.00093}, mr={1802922}}
\end{barticle}
\endbibitem

\bibitem{MR1243240}
\begin{bincollection}[mr]
\bauthor{\bsnm{Tang},~\bfnm{Shan~Jian}\binits{S.~J.}} \AND
  \bauthor{\bsnm{Li},~\bfnm{Xun~Jing}\binits{X.~J.}}
(\byear{1994}).
\btitle{Maximum principle for optimal control of distributed parameter
  stochastic systems with random jumps}.
In \bbooktitle{Differential Equations, Dynamical Systems, and Control Science}.
\bseries{Lecture Notes in Pure and Appl. Math.}
\bvolume{152}
\bpages{867--890}.
\bpublisher{Dekker}, \baddress{New York}.
\bid{mr={1243240}}
\end{bincollection}
\endbibitem

\bibitem{MR1440146}
\begin{barticle}[mr]
\bauthor{\bsnm{Yong},~\bfnm{Jiongmin}\binits{J.}}
(\byear{1997}).
\btitle{Finding adapted solutions of forward--backward stochastic differential
  equations: Method of continuation}.
\bjournal{Probab. Theory Related Fields}
\bvolume{107}
\bpages{537--572}.
\bid{doi={10.1007/s004400050098}, mr={1440146}}%
\end{barticle}%
\endbibitem%

\bibitem{MR1696772}
\begin{bbook}[vtex]
\bauthor{\bsnm{Yong},~\bfnm{Jiongmin}\binits{J.}} \AND
  \bauthor{\bsnm{Zhou},~\bfnm{Xun~Yu}\binits{X.~Y.}}
(\byear{1999}).
\btitle{Stochastic Controls: Hamiltonian Systems and HJB Equations}.
\bseries{Applications of Mathematics (New York)}
\bvolume{43}.
\bpublisher{Springer}, \baddress{New York}.
\bid{mr={1696772}}
\end{bbook}
\endbibitem

\end{thebibliography}
\end{document}